\documentclass[11pt,leqno]{amsart}
\makeatletter
\usepackage{amsfonts,delarray,amssymb,amsmath,amsthm,a4,a4wide}
\usepackage{color}
\usepackage{marginnote}
\usepackage[margin=0.88in]{geometry}

\newcommand{\review}[2][\right]{\relax
\ifx#1\right\relax \left.\fi#2#1\rvert}
\let\abs=\envert

\newtheorem{thm}{Theorem}[section]

\newtheorem{rem}[thm]{Remark}
\newtheorem{cor}[thm]{Corollary}
\newtheorem{lemma}[thm]{Lemma}
 \newtheorem{definition}[thm]{Definition}

\newcommand{\bef}{\begin{flushright}}
\newcommand{\eef}{\end{flushright}}

\newcommand{\eval}[2][\right]{\relax
\ifx#1\right\relax \left.\fi#2#1\rvert}

\let\abs=\envert

\numberwithin{equation}{section}

\newcommand{\R}{\mathbb{R}}

\renewcommand{\div}{\mbox{div}\,}
\newcommand{\p}{\partial}

\newcommand\e{\varepsilon}  
\newcommand{\h}{\hspace*{.24in}}

\def\h{\hspace*{.24in}}

\def\beq{\begin{eqnarray*}}
\def\eeq{\end{eqnarray*}}

\def\RR{\mbox{$I\hspace{-.06in}R$}}

\newenvironment{myindentpar}[1]%
{\begin{list}{}%
         {\setlength{\leftmargin}{#1}}%
         \item[]%
}
{\end{list}}

\begin{document}

\title[Second Variations for Allen-Cahn Energies and Neumann Eigenvalues] {Asymptotic Behavior of Allen-Cahn Type Energies and Neumann Eigenvalues via Inner Variations}
\author{Nam Q. Le}
\address{Department of Mathematics, Indiana University, Bloomington, 831 E 3rd St,
Bloomington, IN 47405, USA. }
\email {nqle@indiana.edu $\qquad$ sternber@indiana.edu}
\author{Peter J. Sternberg}
\thanks{The research of the first author was supported in part by NSF grants DMS-1500400 and DMS-1764248.  The second author was supported by NSF grant DMS-1362879}

\maketitle

\begin{abstract}
We use the notion of first and second inner variations as a bridge allowing one to pass to the limit of first and second Gateaux variations for the  Allen-Cahn, Cahn-Hilliard and Ohta-Kawasaki energies. Under suitable assumptions, this allows us to show that stability passes to the sharp interface limit, including boundary terms, by considering non-compactly supported velocity and acceleration fields in our variations. This complements the results of 
Tonegawa, and Tonegawa and Wickramasekera, where interior stability is shown to pass to the limit. As a further application,
we prove an asymptotic upper bound on the $k^{th}$ Neumann eigenvalue of the linearization of the Allen-Cahn operator, relating it to the $k^{th}$ Robin eigenvalue of the Jacobi operator, taken with respect to the minimal surface arising as the asymptotic location of the zero set of the Allen-Cahn critical points. We also prove analogous results for eigenvalues of the linearized operators arising in the Cahn-Hilliard and Ohta-Kawasaki settings. 
These complement the earlier result of the first author where such an asymptotic upper bound is achieved for Dirichlet eigenvalues for the linearized Allen-Cahn operator. Our asymptotic upper bound on Allen-Cahn Neumann eigenvalues extends, in one direction, the asymptotic {\it equivalence} of these eigenvalues established  in the work of Kowalczyk in the two-dimensional case where the minimal surface is a line segment and specific Allen-Cahn critical points are suitably constructed.
\end{abstract}
\pagenumbering{arabic}

\section{Introduction and Statements of the Main Results}  \setcounter{equation}{0}  

Within the calculus of variations, the second variation is of course a powerful tool in analyzing the nature of critical points. This is in particular the case in the context of energetic models involving double-well potentials perturbed by a gradient penalty term such as the Allen-Cahn or Modica-Mortola, Cahn-Hilliard and Ohta-Kawasaki functionals \cite{AC,OK}. As the scale of interfacial energy approaches zero, these energy functionals all converge, in the sense of  $\Gamma$-convergence, to a variety of sharp interface models and there are many studies of critical points associated with these energies or with their $\Gamma$-limits for which the second variation plays a crucial role. Taking a limit of the second variations themselves to obtain the second variation of the $\Gamma$-limit, however, can be problematic and the results in this direction are far fewer. Here, building on the techniques and results found in \cite{Le, Le2}, we carry out this limiting process using the notion of inner variation, to be defined precisely in Section \ref{Var_sec}. The inner variation provides a bridge between the second variations of the so-called diffuse models listed above and those of the sharp interface variational problems arising as their $\Gamma$-limits which tend to involve minimal or constant mean curvature hypersurfaces.  For more on $\Gamma$-convergence, we refer to \cite{Braides} or \cite{DM}. Its definition for the Allen-Cahn functional will be briefly recalled in Section \ref{AC_sec}. 

In \cite{Le, Le2}, the first author passes to the limit in second variations of various energies including the Allen-Cahn functional 
\begin{equation}
E_{\varepsilon}(u):=\int_{\Omega}\left(\frac{\varepsilon \abs{\nabla u}^2}{2} +\frac{(1-u^2)^2}{2\varepsilon}\right) dx,\quad u:\Omega\to\R,\;\Omega\subset\R^N \;(N\geq 2),\label{ACintro}
\end{equation}
in the context of critical points $u_\e$, that is $u_\e$ satisfying $-\e\Delta u_\e + 2\e^{-1} (u_\e^3-u_\e)=0$ in $\Omega$, subject to Dirichlet boundary conditions. This leads, in particular, to an asymptotic upper bound on the Dirichlet eigenvalues, namely
\begin{equation}
\limsup_{\e\to 0}\frac{\lambda_{\e,k}}{\e}\leq \lambda_k\quad\mbox{for}\;k=1,2,\ldots\label{Devals}
\end{equation}
where $\lambda_{\e,k}$ denotes the $k^{th}$ Dirichlet eigenvalue of the linearized Allen-Cahn operator 
\[-\e\Delta+\frac{2}{\e}(3u_\e^2-1),
\] subject to zero boundary conditions on $\partial\Omega$ and $\lambda_k$ denotes the $k^{th}$ eigenvalue of the Jacobi operator $-\Delta_{\Gamma}-\abs{A_\Gamma}^2$ associated with a minimal surface $\Gamma$ subject to zero boundary conditions on $\partial\Gamma$. Here $\Gamma$ denotes the asymptotic location of the interfacial layer bridging $\{u_\e\approx 1\}$ and $\{u_\e\approx -1\}$ and $A_\Gamma$ denotes the associated second fundamental form.
This particular result in \cite{Le2} (see Corollary 1.1 there) has been recently extended to the closed Riemannian setting in \cite{Ga} by Gaspar who also relaxed  the multiplicity 1 assumption in \cite{Le2}; see also Hiesmayr \cite{Hi}
for related results. 
Related to such results on the Dirichlet problem is the elegant work in \cite{Tonegawa,TW}, where the authors show within the context of varifolds that when stable critical points of the Allen-Cahn functional converge to a limit, the 
limiting interface is stable with respect to interior perturbations; moreover, the limiting interface is smooth in dimensions $N\leq 7$ while its singular set (if any) has Hausdorff dimension at most $N-8$ in dimensions
$N>7$. We would like to emphasize that the convergence and regularity results in \cite{Tonegawa, TW} rely on an important  interior convergence result for the Allen-Cahn equation from the work of Hutchinson-Tonegawa \cite{HT}
and a deep interior regularity theory for stable codimension 1 integral varifolds  from the work of  Wickramasekera \cite{W}. At present, to the best of our knowledge, there are no boundary analogues for the above results.

In this article we extend the techniques of \cite{Le, Le2} in three directions: we allow for a mass constraint so as to cover not just the Allen-Cahn context but also Cahn-Hilliard, we allow for perturbation by a nonlocal term as arises in the Ohta-Kawasaki functional, \eqref{OKintro}, and most crucially, we consider non-compactly supported variations of domain in taking inner variations, allowing us to capture boundary effects in passing to the limit in the case of  
Neumann boundary conditions in all of these problems.

 Regarding this last extension, we point out that the ``natural'' Neumann boundary conditions satisfied by critical points in all of these models are {\it not} the boundary conditions associated with the limit. Rather, for example, in the case of Allen-Cahn energy, the analogue of the result \eqref{Devals} from \cite{Le2} is that \eqref{Devals} holds for $\lambda_{\e,k}$ associated with homogeneous Neumann boundary conditions but for
$\lambda_k$ associated with Robin boundary conditions, cf. \eqref{Robinintro}. For two-dimensional Allen-Cahn, this shift from Neumann for $\e>0$ to Robin in the limit is examined in detail by Kowalczyk in \cite{K} where it is shown that $$\lim_{\e\to 0}\frac{\lambda_{\e,k}}{\e}= \lambda_k$$ for a carefully constructed sequence of Neumann critical points $\{u_\e\}$ and so for that problem our results represent a one-sided generalization to a more general class of critical points and to arbitrary dimensions.

In the next section we will give a precise definition of first and second inner variations while reviewing the more standard notion of first and second Gateaux variations. Roughy speaking, though, the difficulty in transitioning from the second Gateaux variation $d^2 E_{\e}(u_\e,\varphi)$ of a functional like $E_\e$ in \eqref{ACintro} to that of its limit, say $E(\Gamma)$, which is essentially area or $(N-1)$-dimensional Hausdorff measure $\mathcal{H}^{N-1}(\Gamma)$, is that the former is computed by  taking the second $t$-derivative of $E_{\e} (u_\e+t\varphi)$ evaluated at $t=0$
where $\varphi$ is a scalar function, while the latter comes from taking the second $t$-derivative of $\mathcal{H}^{N-1}\big(\Phi_t (\Gamma)\big)$ evaluated at $t=0$ where  $\Phi_t$ is a deformation of the identity map of the form
\[
\Phi_t(x)\sim x+t\eta(x)+\frac{t^2}{2} \zeta(x)
\] for some velocity and acceleration vector fields $\eta$ and $\zeta$ mapping $\R^N\to\R^N$. A successful  passage from one of these variations to the other, however, should be computed by similar methods. Bridging these two disparate notions is the inner variation.
Indeed, if we view $\Gamma$ as the asymptotic location of the zero level set of $u_\e$, and if we view $\Phi_t$ as a deformation not just of $\Gamma$ but of all points in $\R^N$, then $\Phi_t(\Gamma)$ corresponds to the limit of 
the zero level of $u_\e(\Phi^{-1}_t(x))$. Thus, we might be led to compute the  first and second $t$-derivatives of $E_\e(u_\e(\Phi^{-1}_t(x)))$, and these are precisely the inner variations. Then relating these quantities to the more standard first and second Gateaux variations becomes one of our first tasks. 

Differently put, inner variation allows us to more directly compare the energy landscapes of diffuse models and their sharp interface limits. In the present paper we carry out this explicit bridging for the Allen-Cahn functional
as well as its nonlocal counterpart, the Ohta-Kawasaki functional, where the limiting object is a hypersurface, but we would like to point out that examples of this bridging via inner variations already exists in the literature, especially in the Ginzburg-Landau setting, where limiting objects are instead finite sets of points in planar domains, namely Ginzburg-Landau vortices. This includes Serfaty's stability analysis in \cite{Serfaty}, as well as  \cite{SS1} (see also \cite{Serfaty2}), where Sandier and Serfaty 
introduce a powerful $\Gamma$-convergence of gradient flows scheme
in which they identify certain energetic conditions between the $\Gamma$-converging functionals and their $\Gamma$-limits that guarantee convergence of their corresponding gradient flows. When applied to Ginzburg-Landau vortices which lie in the interior of the planar domain sample, 
 the verification of one of the two key sufficient conditions is done by a constructive argument using inner variations with compactly supported vector fields; see \cite[equation (3.27)]{SS1}. For boundary vortices in thin magnetic films, this verification is carried out by Kurzke \cite{Kurzke} using inner variations with non-compactly supported vector fields; see \cite[Theorem 6.1]{Kurzke}.

Along with giving the definitions of first and second inner variations, and reviewing the definitions of Gateaux variations, establishing this relationship between the two notions of variation is the content of Section \ref{Var_sec}. In Section 3 we pass to the limit
in the inner variations of the Allen-Cahn functional; see Theorem \ref{thm-AC2}. The proof relies crucially on a convergence result of Reshetnyak \cite{Res} stated in a convenient form from Spector \cite{Sp} in Theorem \ref{Sp_thm}. In Section 4 we present two applications of Theorem \ref{thm-AC2}. The first, Theorem \ref{ACstab},
shows that under suitable regularity hypotheses on the limiting interface, stability of Allen-Cahn critical points passes to the limit. Thus, in the limit we recover the second variation formula including boundary terms derived in \cite{SZ2}. The second is the previously alluded to generalization of \eqref{Devals} to the Neumann setting which we state here as our first main result:
\begin{thm}[Upper semicontinuity of the Allen-Cahn Neumann eigenvalues]
\label{eigen_thm}
 Let $\Omega$ be an open smooth bounded domain in $\RR^{N}$ ($N\geq 2$).
Let $\{u_{\e}\}\subset C^3 (\overline{\Omega})$ be a sequence of critical points of the Allen-Cahn functional \eqref{ACintro}
that converges in $L^1(\Omega)$ to a function $u_0\in BV (\Omega, \{1, -1\})$ with an interface $\Gamma:=\partial\{u_0=1\}\cap\Omega$ having the property that $\overline{\Gamma}$ is $C^2$. 
Assume that
 $\lim_{\varepsilon \rightarrow 0} E_{\varepsilon}(u_{\varepsilon}) = 
\frac{4}{3}\mathcal{H}^{N-1}(\Gamma)$, and assume
that $\Gamma$ is connected. Let $\lambda_{\e, k}$ be the $k^{th}$ eigenvalue of the operator $-\e \Delta + 2\e^{-1}(3u_{\e}^2-1)$
 in $\Omega$ with zero Neumann condition on $\partial\Omega$. Let 
$\lambda_k$ and $\varphi^{(k)}:\overline{\Gamma}\to\R$ be the $k^{th}$ eigenvalue and eigenfunction of the operator $-\Delta_{\Gamma} - |A_{\Gamma}|^2$ in $\Gamma$ subject to Robin boundary conditions on $\partial\Gamma \cap\partial\Omega$, namely
\begin{equation}
 \label{Robinintro}
 \left\{
 \begin{alignedat}{2}
  (-\Delta_{\Gamma} - |A_{\Gamma}|^2)\varphi^{(k)} &=\lambda_k \varphi^{(k)}~&&\text{in} ~\Gamma, \\\
 \frac{\partial\varphi^{(k)}}{\p  {\bf n}}+A_{\partial\Omega}({\bf n}, {\bf n}) \varphi^{(k)} &=0\h~&&\text{on}~\partial\Gamma\cap\partial\Omega.
 \end{alignedat}
 \right.
\end{equation}
Here ${\bf n} = (n_{1},\cdots,n_{N})$ denotes the unit normal to $\Gamma$ pointing out of the region $\{x\in\Omega:\,u_0(x)=1\}$ and $A_{\Gamma}$ and $A_{\p\Omega}$ denote the second fundamental forms of $\Gamma$ and $\p\Omega$, respectively. Then
\[
\limsup_{\e\rightarrow 0}\frac{\lambda_{\e, k}}{\e}\leq \lambda_k.
\]
\end{thm}
The proof of Theorem \ref{eigen_thm} will be given in Section 4.
\vglue 0.1cm
\noindent
We mention that when $\Gamma$ is a minimal hypersurface satisfying certain nondegeneracy conditions, Pacard and Ritor\'{e} \cite{PR} construct critical points $u_{\e}$ of $E_{\e}$ whose zero 
level sets converge to $\Gamma$ and the limit $\lim_{\varepsilon \rightarrow 0} E_{\varepsilon}(u_{\varepsilon}) = 
\frac{4}{3}\mathcal{H}^{N-1}(\Gamma)$ holds. Thus, Theorem \ref{eigen_thm} applies in particular to this case. Also we should say that we do not know whether there are contexts beyond the previously mentioned planar result in \cite{K} where asymptotic {\it equality} holds rather than just inequality.

In Sections \ref{B_sec} and \ref{OK_sec} we extend our study to the Ohta-Kawasaki functional which involves a nonlocal term:
\begin{equation}
\label{OKintro}
\mathcal{E}_{\e,\gamma} (u) =\int_{\Omega}\left(\frac{\varepsilon \abs{\nabla u}^2}{2} +\frac{(1-u^2)^2}{2\varepsilon}\right) dx +\frac{4}{3}\gamma\int_{\Omega}\int_{\Omega} G(x,y)u(x) u(y) dx dy
\end{equation}
where $\gamma\geq 0$ is a fixed constant and
$G(x,y)$ is the Green's function for $\Omega$ satisfying $$-\Delta G =\delta -\frac{1}{|\Omega|} ~\text{on } \Omega$$ with Neumann boundary condition.
We associate 
 to each $u\in L^2(\Omega)$ a function $v\in W^{2,2}(\Omega)$, denoted by $(-\Delta)^{-1} u$, as the solution to the following Poisson equation with Neumann boundary condition:
$$
-\Delta v = u-\frac{1}{|\Omega|}\int_{\Omega} u dx~\text{in}~\Omega, \frac{\partial v}{\partial \nu}=0~\text{on}~\partial\Omega, ~\int_{\Omega} v(x) dx=0.
$$
Note that
$$(-\Delta)^{-1} u = \int_{\Omega} G(x, y) u(y) dy.$$
Let us denote the second inner variation of $\mathcal{E}_{\e, \gamma}$ at $u_\e$ with respect to $C^3(\overline{\Omega})$ vector fields
$\eta,\zeta$ by
$$\delta^{2} \mathcal{E}_{\e,\gamma}(u_{\varepsilon}, \eta, \zeta):=  \left.\frac{d^2}{dt^2}\right\rvert_{t=0}  \mathcal{E}_{\e,\gamma}\left(u_\e\circ (I +t\eta + \frac{t^2}{2}\zeta)^{-1}\right).$$
A more comprehensive analysis concerning inner variations will be presented in Section \ref{Var_sec}.

Our second main result is summarized in the following theorem.

\begin{thm}[Stability of Ohta-Kawasaki passes to the limit; upper semicontinuity of Ohta-Kawasaki eigenvalues]
\label{OK2}
Let $\Omega$ be an open smooth bounded domain in $\RR^{N}$ ($N\geq 2$). 
Let $\gamma\geq 0$. Fix $m\in (-1,1)$.
Let $\{u_{\e}\}\subset C^3 (\overline{\Omega})$ be a sequence of critical points of the Ohta-Kawasaki functional (\ref{OKintro})
 subject to the mass constraint
$
\frac{1}{|\Omega|} \int_{\Omega} u\,dx = m
$
that converges in $L^2(\Omega)$ to a function $u_0\in BV (\Omega, \{1, -1\})$ with an interface $\Gamma=\partial\{u_0=1\}\cap\Omega$ having the property that $\overline{\Gamma}$ is $C^2$. 
Assume that
 $$\frac{3}{4}\lim_{\varepsilon \rightarrow 0} \mathcal{E}_{\e,\gamma}(u_{\varepsilon}) = \mathcal{E}_{\gamma}(\Gamma):=
\mathcal{H}^{N-1}(\Gamma) + \gamma \int_{\Omega}\int_{\Omega} G(x,y)u_0(x) u_0(y) dx dy.$$
Let $v_0(x)=  \int_{\Omega} G(x, y) u_0(y) dy$. For any smooth function $\xi:\overline{\Omega}\rightarrow\RR$, we denote
\begin{multline*}
\delta^2 \mathcal{E}_{\gamma} (\Gamma,\xi):=\int_{\Gamma} \left(|\nabla_{\Gamma}\xi|^2  -|A_{\Gamma}|^2\xi^2\right)\, d\mathcal{H}^{N-1} - \int_{\partial\Gamma\cap\partial\Omega} A_{\partial\Omega}({\bf n}, {\bf n})\xi^2 \,d\mathcal{H}^{N-2}\nonumber \\ +  8\gamma\int_{\Gamma}\int_{\Gamma} G(x, y) \xi(x)\xi(y) d \mathcal{H}^{N-1} (x) d \mathcal{H}^{N-1}(y) + 4 \gamma\int_{\Gamma} (\nabla v_0\cdot {\bf n}) \xi^2 d \mathcal{H}^{N-1} (x).
\end{multline*}
Here ${\bf n} = (n_{1},\cdots,n_{N})$ denotes the unit normal to $\Gamma$ pointing out of the region $\{x\in\Omega:\,u_0(x)=1\}$.
Then,
the following conclusions hold:
\begin{myindentpar}{1cm}
(i)  There is a constant $\lambda$ such that $(N-1) H + 4 \gamma v_0 =\lambda$ on $\Gamma$ where $H$ is the mean curvature of $\Gamma$. Moreover, $\p\Gamma$ must meet $\p\Omega$ orthogonally (if at all).\\
(ii) Let $\xi:\overline{\Omega}\rightarrow\RR$ be any smooth function satisfying
$
\int_{\Gamma} \xi(x) d\mathcal{H}^{N-1}(x)=0.
$
Then, for all smooth vector fields $\eta\in (C^{3}(\overline{\Omega}))^{N}$ with $\eta=\xi {\bf n}$ on $\Gamma$, $\eta\cdot \nu=0$ on $\partial\Omega$,  $({\bf n},{\bf n}\cdot\nabla\eta) =0$  on $\Gamma$ and for $W:= (\eta \cdot\nabla) \eta-(\div \eta)\eta$, we 
have
\begin{align}\frac{3}{4}\lim_{\varepsilon\rightarrow 0}\delta^{2} \mathcal{E}_{\e,\gamma}(u_{\varepsilon}, \eta, W) = \delta^2 \mathcal{E}_{\gamma} (\Gamma,\xi).
\label{secst}
\end{align} 
(iii) If $\{u_\e\}$ are stable critical points of $\mathcal{E}_{\e,\gamma}$ with respect to the mass constraint $\frac{1}{|\Omega|} \int_{\Omega} u\, dx = m$, then 
for all smooth function $\xi:\overline{\Omega}\rightarrow\RR$ satisfying
$
\int_{\Gamma} \xi(x) d\mathcal{H}^{N-1}(x)=0,
$
we have
$$\delta^2 \mathcal{E}_{\gamma} (\Gamma,\xi)\geq 0.$$
(iv) Assume that $\Gamma$ is connected. Let $\lambda_{\e, \gamma, k}$ be the $k^{th}$ eigenvalue of the operator $-\e \Delta + 2\e^{-1}(3u_{\e}^2-1) +\frac{8}{3} \gamma(-\Delta)^{-1}$ in $\Omega$ with zero Neumann condition on $\partial\Omega$. Let 
$\lambda_{\gamma, k}$ and $\varphi^{(\gamma, k)}:\overline{\Gamma}\to\R$ be the $k^{th}$ eigenvalue and eigenfunction of the operator $-\Delta_{\Gamma} - |A_{\Gamma}|^2 + 8\gamma (-\Delta)^{-1}(\chi_{\Gamma})+ 4\gamma (\nabla v_0\cdot {\bf n})$ in $\Gamma$ subject to Robin boundary conditions on $\partial\Gamma \cap\partial\Omega$, namely
\begin{equation*}
\small
 \left\{
 \begin{alignedat}{2}
  \left(-\Delta_{\Gamma} - |A_{\Gamma}|^2 + 4\gamma (\nabla v_0\cdot {\bf n})\right)\varphi^{(\gamma, k)} (x) + 8\gamma \int_{\Gamma} G(x, y) \varphi^{(\gamma, k)} (y)d \mathcal{H}^{N-1} (y)  &=\lambda_{\gamma, k}\varphi^{(\gamma, k)}(x)~&&\text{in} ~\Gamma, \\\
 \frac{\partial\varphi^{(\gamma, k)}}{\p  {\bf n}}+A_{\partial\Omega}({\bf n}, {\bf n}) \varphi^{(\gamma, k)} &=0\h~&&\text{on}~\p \Gamma\cap\partial\Omega.
 \end{alignedat}
 \right.
\end{equation*}
Then
\[
\limsup_{\e\rightarrow 0}\frac{\lambda_{\e, \gamma, k}}{\e}\leq \lambda_{\gamma, k}.
\]
(v) The conclusion in (iv) also holds if in the above eigenvalue problems we replace the homogeneous Neumann conditions and Robin boundary conditions by homogeneous Dirichlet boundary conditions.
\end{myindentpar}
\end{thm}
The proof of Theorem \ref{OK2} will be given in Section \ref{OK_sec}.
\vglue 0.1cm
\noindent
Item (i) in Theorem \ref{OK2} above is just the condition of criticality for the limiting functional $\mathcal{E}_\gamma$ while the right-hand side of \eqref{secst}, that is $\delta^2 \mathcal{E}_{\gamma} (\Gamma,\xi)$, is its second variation (see \cite[Theorems 2.3 and 2.6]{CS}), so item (iii) of the theorem asserts that  stability is passed to the limiting interface. A special case of Theorem \ref{OK2} (iv) where $\gamma=0$ is an extension of our Theorem \ref{eigen_thm} on the Allen-Cahn functional to the mass-constrained Cahn-Hilliard setting. 

We should say that throughout this article we have not sought to present results under weakest possible regularity assumptions on the limiting interface. Adapting results to the situation where the limiting interface possesses a low-dimensional singular set should be feasible.

\subsection*{Notation}
 Throughout, $\Omega$ is an open, smooth, bounded domain in $\RR^{N}$ ($N\geq 2$).
We let $\nu$ be the outer unit normal to $\p\Omega$. For any Lebesgue measurable subset $S\subset \R^N$, we use $|S|$ to denote its $N$-dimensional Lebesgue measure.
If $F: \RR\times \RR^N\rightarrow \RR$ is a smooth function then 
we will write
$F= F(z, {\bf p})$
for $z\in \RR$ and ${\bf p} =(p_1, \cdots, p_N)\in \RR^N$ and we will set $\nabla_{\bf p} F = (F_{p_1}, \cdots, F_{p_N}).$ 
If $\eta:\Omega\rightarrow \RR^N$ is a vector field, then we write $\eta = (\eta^1,\cdots,\eta^N)$. If $\eta\in (C^1 (\Omega))^N$, we define a new vector field $Z:=(\eta\cdot \nabla ) \eta $ whose $i$-th component is 
$
Z^i =  \frac{\partial \eta^i}{\partial x_j}\eta^j,
$
invoking the summation convention on repeated indices.
We use $(\nabla \eta)^2$ to denote the matrix whose $(i, k)$ entry is
$
\frac{\p \eta^i}{\p x_j} \frac{\p \eta^j}{\p x_k},
$
and we use $(\cdot, \cdot)$ to denote the standard inner product in $\RR^N$.

When a differentiable function, say $\phi$, is scalar-valued so that there is no room for confusion, we write $\phi_i=\frac{\p \phi}{\p x_i}$.
\section{The Relationship Between Gateaux and Inner Variations}
\label{Var_sec}

In this section, we first review the definitions of Gateaux variations, then give the definitions of first and second inner variations, and finally establish the relationship between the two notions of variation.

The typical functionals we consider are of the form
\begin{equation}
A(u):= \int_{\Omega} F(u(x), \nabla u(x)) dx\label{basic}
\end{equation}
where $u\in C^3(\overline{\Omega})$ and $F: \RR\times \RR^N\rightarrow \RR$ is a smooth function. We mention that in this paper, for ease of presentation, we state results under very generous regularity conditions on the functions and functionals involved. No doubt many of these smoothness assumptions could be relaxed.
\subsection{Gateaux variations and inner variations}
We recall that the first and second Gateaux variations of $A$ at $u\in C^{3}(\Omega)$ with respect to $\varphi \in C^{3}(\Omega)$, denoted here by $dA(u,\varphi)$ and $d^2 A(u,\varphi)$ respectively, are defined by 
$$
dA(u,\varphi):= \left.\frac{d}{dt}\right\rvert_{t=0} A(u + t\varphi),\quad d^{2} A(u,\varphi) := \left.\frac{d^2}{dt^2}\right\rvert_{t=0} A(u + t\varphi);$$
see, for example, \cite[Chapter 1]{Wi}.

On the other hand, a distinct notion of variation is that of inner variation, usually taken with respect to compactly supported vector fields, see e.g. \cite[pp. 283-293 of Section 3.1.1]{GMS}. 
It has been used in several contexts, for example, in the study of weakly Noether harmonic maps \cite[Section 1.4.2]{H}, in the investigation of the asymptotics for solutions of the Ginzburg-Landau system \cite[Chapter 13]{SS2}, 
and also in second order asymptotic limits in phase transitions \cite{Le, Le2}, to name a few. Most closely related to the subject of this paper are the works
 \cite{Le, Le2} where the first author studies the Morse index and upper semicontinuity
of eigenvalue problems in phase transitions when Dirichlet boundary conditions are enforced. Inspired by the case of compactly supported vector fields, we define below the concept of inner variations with respect to general, that is, not necessarily compactly supported, vector fields, in order to examine the corresponding asymptotics of Neumann eigenvalues.

To this end,
consider any smooth vector field $\eta\in (C^{3}(\overline{\Omega}))^{N}$ and associated with it, suppose that we have a $t$-dependent map $\Phi_t$ with the property that
\begin{equation}
\label{map-deform1} 
\Phi_{t} (x) = x + t\eta(x) + O(t^2).
\end{equation}
In this paper, by $O(t^k)$ $(k\leq 3)$, we mean any quantity $Q(x,t)$ such that it is $C^3$ in the variables $x$ and $t$ and furthermore $\abs{Q(x,t)}/\abs{t}^k$ is uniformly bounded in $\overline{\Omega}$ when $|t|$ is small.

For $\abs{t}$ sufficiently small, the map $\Phi_t$ is a diffeomorphism of $\RR^N$ onto itself and thus we can define its inverse map $\Phi_t^{-1}$.
We then define the first inner variation of $A$  at $u$ with respect to the velocity vector field $\eta$ by 
\begin{equation}
\label{inner1defn}
\delta A(u,\eta) := \left.\frac{d}{dt}\right\rvert_{t=0} A(u\circ\Phi_t^{-1}).
\end{equation}
Now if in addition to $\eta$ we consider a second smooth vector field $\zeta\in (C^{3}(\overline{\Omega}))^{N}$ and if the diffeomorphism $\Phi_t(x)$ satisfies
\begin{equation}
\label{map-deform2} 
\Phi_{t} (x) = x + t\eta(x) + \frac{t^2}{2}\zeta(x)+O(t^3),
\end{equation}
then we define the second inner variation of $A$ at $u$ with respect to the velocity vector field $\eta$ and acceleration vector field $\zeta$ by
\begin{equation}
 \delta^{2} A(u,\eta,\zeta) := \left.\frac{d^2}{dt^2}\right\rvert_{t=0} A(u\circ\Phi_t^{-1}).\label{innerdefn2}
\end{equation}
We note that $\Phi_t^{-1}$ does not map $\Omega$ to $\Omega$ in general. Thus, in calculating inner variations, we implicitly extend $u$ to be a smooth function on a neighborhood of $\overline{\Omega}$. The calculations show that the inner variations do not depend on these extensions.
\begin{rem}
In the above definitions of variations, we do not use any particular form of $A$. Thus, they apply equally to local functionals of the form (\ref{basic})
and nonlocal functionals of the form (\ref{B_def}) in Section \ref{B_sec}.
\end{rem}
The goal of the next subsection is to calculate the above variations and to explore their relationship.
\subsection{Calculation and relationship between variations} Let $A$ be as in (\ref{basic}).
Carrying out the standard computation of $\left.\frac{d}{dt}\right\rvert_{t=0} A(u + t\varphi)$ and $\left.\frac{d^2}{dt^2}\right\rvert_{t=0} A(u + t\varphi)$ for $u$ and $\phi$ in $C^1(\Omega)$, we obtain the well-known formulas for the first and second Gateaux variations:
\begin{eqnarray}
\label{fveq}
dA(u,\varphi)=\left.\frac{d}{dt}\right\rvert_{t=0} A(u + t\varphi)=\int_{\Omega}\left(F_{z}\varphi + F_{p_i}\varphi_i\right) dx.
\end{eqnarray}
and
\begin{equation}\label{sveq} 
d^2 A(u,\varphi)=\left.\frac{d^2}{dt^2}\right\rvert_{t=0} A(u + t\varphi)=\int_{\Omega}\left( F_{zz}\varphi^2 + 2 F_{z p_i} \varphi\varphi_i + F_{p_i p_j}\varphi_i \varphi_j\right) dx,
\end{equation}
where in these formulae all derivatives of $F$ are evaluated at $z=u$ and ${\bf p}=\nabla u.$ 

We turn now to the calculation of inner variations. In the following lemmas, we establish two different formulas for the inner variations of the functional $A$. The first is more general and is obtained via direct calculation. The second we prove via a change of variables.
These formulas will be used in our proof of the asymptotic upper bound for Allen-Cahn Neumann eigenvalues.
\begin{lemma} [Inner variations via direct calculation]
\label{SIV_direct} Let $A$ be as in (\ref{basic}).
Assume that $u\in C^3(\overline{\Omega})$. Let $\eta, \zeta\in (C^{3}(\overline{\Omega}))^{N}$.
The first inner variation of $A$  at $u$ with respect to $\eta$ is given by
\[
 \delta A(u,\eta)= \int_{\Omega} \left[F_z (-\nabla u\cdot\eta ) + F_{p_i} (\frac{\partial}{x_i}(-\nabla u\cdot \eta))\right] dx.
 \]
The second inner variation of $A$  at $u$ with respect to $\eta$ and $\zeta$ is
\begin{multline*} \delta^2 A(u,\eta,\zeta) =\int_{\Omega} \left[F_{zz} (\nabla u\cdot \eta)^2 + 2 F_{zp_i} (\nabla u\cdot\eta)\frac{\partial}{x_i}(\nabla u\cdot \eta) + F_{p_i p_j} \frac{\partial}{x_i} (\nabla u \cdot \eta)\frac{\partial}{x_j}(\nabla u \cdot \eta)
\right] dx\\ + \int_{\Omega} \left[F_{z} X_0 + F_{p_i} \frac{\partial}{x_i}X_0\right] dx,
\end{multline*}
where $X_0$ is given by 
\begin{equation}
\label{Xzero}
X_0:= (D^2 u \cdot \eta, \eta) + (\nabla u, 2(\eta\cdot \nabla ) \eta -\zeta).
\end{equation}
\end{lemma}
In view of \eqref{fveq}, it then immediately follows that:
\begin{cor}
\label{inner_rem}
Let $A$ be as in (\ref{basic}).
If $u\in C^3(\overline{\Omega})$ and $\eta, \zeta\in (C^{3}(\overline{\Omega}))^{N}$,
then one has
\begin{eqnarray}
&&\delta A(u,\eta,\zeta)= dA (u, -\nabla u\cdot \eta),\label{first}\\
&&\delta^2 A(u,\eta,\zeta)= d^2A (u, -\nabla u\cdot \eta) + dA (u, X_0),
\label{secondgen}
\end{eqnarray}
and if $u$ is a critical point of $A$, that is, if $dA (u,\varphi)=0$ for all $\varphi \in C^3 (\overline{\Omega})$, then 
$\delta^2 A(u,\eta,\zeta)$ is independent of $\zeta$. Moreover, in this case, 
\begin{equation}
\delta^2 A(u,\eta,\zeta)= d^2A (u, -\nabla u\cdot \eta)\quad\mbox{for all }\eta, \zeta\in (C^{3}(\overline{\Omega}))^{N}.\label{seconde}
\end{equation}
\end{cor}

\begin{lemma}[Inner variations for velocity vector fields tangent to the domain boundary] 
\label{SIV_ch}
Let $A$ be as in (\ref{basic}).
Assume that $u\in C^3(\overline{\Omega})$. Suppose that $\eta\in (C^{3}(\overline{\Omega}))^{N}$
where $\eta\cdot \nu=0$ on $\p\Omega$. 
The first inner variation of $A$  at $u$ with respect to $\eta$ is
\begin{equation*}
 \delta A(u,\eta)=
\int_{\Omega}\left\{ F \div \eta - (\nabla_{\bf p} F, \nabla u\cdot \nabla\eta) \right\} dx.
\end{equation*}
The second inner variation of $A$  at $u$ with respect to $\eta$ and $Z:= (\eta\cdot\nabla)\eta$ is
\begin{multline*}
\delta^2 A(u,\eta, Z)=
\int_{\Omega}\left\{ FX - 2 (\nabla_{\bf p} F, \nabla u\cdot \nabla\eta) \,{\rm div} \,\eta - 2 (\nabla_{\bf p} F, Y) + F_{p_i p_j}(\nabla u\cdot\nabla \eta)^{i}(\nabla u\cdot\nabla \eta)^{j}\right\} dx.
\end{multline*}
where
\begin{equation}
\label{XY_eq}
X:= {\rm div}\, Z + ({\rm div}\, \eta)^2 -{\rm trace} \,(\nabla\eta)^2;~~Y= \frac{1}{2} \nabla u\cdot\nabla Z-(\nabla \eta)^2\cdot \nabla u.
\end{equation}
\end{lemma}

\begin{rem}
\label{weird}
In light of the fact that the formula for the second inner variation in Lemma \ref{SIV_ch} is a special case of the general second inner variation 
$\delta^2 A(u,\eta,\zeta)$ in the case where  $\eta\cdot \nu=0$ on $\p\Omega$ and $\zeta= (\eta\cdot\nabla)\eta$, it follows that if one imposes this boundary condition on $\eta$ and this choice of $\zeta$ in the formula for $\delta^2 A(u,\eta,\zeta)$ given in Lemma \ref{SIV_direct}, then it must be equivalent to the formula given in Lemma \ref{SIV_ch}. We note, however, that it does not seem easy to directly verify this equivalence.
\end{rem}

\begin{rem}
\label{rederive}
We would like to point out that the formulae for inner variations in Lemmas \ref{SIV_direct} and \ref{SIV_ch} already appeared in the proof of \cite[Proposition 2.1]{Le2} for compactly supported vector fields
$\eta$ and $\zeta$. The proof of Lemma \ref{SIV_direct} here follows the same line of argument as in \cite{Le2}. Since it is short and to avoid confusion when adapting to our general vector fields, we include it for the reader's convenience. The proof of Lemma \ref{SIV_ch} is a bit different, utilizing the ODE \eqref{ode} to build the diffeomorphism of $\Omega$. 
\end{rem}

The rest of this section will be devoted to proving Lemmas \ref{SIV_direct} and \ref{SIV_ch}.
\begin{proof}[Proof of Lemma \ref{SIV_direct}]
Let $u_t(y) = u(\Phi_t^{-1}(y))$
where
$$\Phi_{t} (x) = x + t\eta(x) + \frac{t^2}{2}\zeta(x).$$
The formulae are based on the following formula (see \cite[equation (2.16)]{Le2})
\begin{equation}u_t(y) = u (y) -t\nabla u \cdot \eta + \frac{t^2}{2} X_0 + O(t^3).
\label{ut_expand}
\end{equation}

We observe, using (\ref{inner1defn}), (\ref{innerdefn2}) and (\ref{ut_expand}), that the first and second inner variations are equal to the first and second derivatives, respectively, of the following function at $0$:
$$\mathcal{A}_{1}(t) = \int_{\Omega} F(u-t \nabla u\cdot\eta + \frac{t^2}{2}X_0, \nabla u - t\nabla (\nabla u\cdot\eta) +\frac{t^2}{2}\nabla X_0) dx.$$
We compute
$$\mathcal{A}_{1}^{'}(t) =\int_{\Omega}\left [F_z (-\nabla u\cdot\eta + t X_0)-F_{p_i} (\frac{\partial}{x_i}(\nabla u\cdot \eta)- t \frac{\partial}{x_i} X_0)\right]dx$$
and
\begin{multline*}\mathcal{A}_{1}^{''}(t) = \int_{\Omega} \left[F_{zz} (-\nabla u\cdot\eta + t X_0)^2-2F_{zp_i} (\frac{\partial}{x_i}(\nabla u\cdot \eta)- t \frac{\partial}{x_i} X_0)(-\nabla u\cdot\eta + t X_0) \right]dx\\+ \int_{\Omega}\left[F_{p_i p_j} (\frac{\partial}{x_i}(\nabla u\cdot \eta)- t \frac{\partial}{x_i} X_0)(\frac{\partial}{ x_j}(\nabla u\cdot \eta)- t \frac{\partial}{x_j} X_0)\right]dx+ \int_{\Omega} \left[F_{z} X_0 + F_{p_i} \frac{\partial}{x_i} X_0\right].
\end{multline*}
It follows that
\begin{equation*}
 \delta A(u,\eta)=
\mathcal{A}_{1}^{'}(0)= \int_{\Omega} \left[F_z (-\nabla u\cdot\eta ) + F_{p_i} (\frac{\partial}{x_i}(-\nabla u\cdot \eta))\right]
\end{equation*}
and 
\begin{multline*} \delta^2 A(u,\eta,\zeta) =\mathcal{A}_{1}^{''}(0) =\int_{\Omega} \left[F_{zz} (\nabla u\cdot \eta)^2 + 2 F_{zp_i} (\nabla u\cdot\eta)\frac{\partial}{x_i}(\nabla u\cdot \eta)\right] dx\\+ \int_{\Omega}\left[F_{p_i p_j} \frac{\partial}{x_i} (\nabla u \cdot \eta)\frac{\partial}{x_j}(\nabla u \cdot \eta)\right] dx+ \int_{\Omega} \left[F_{z} X_0 + F_{p_i} \frac{\partial}{x_i}X_0\right]dx.
\end{multline*}

\end{proof}

\begin{proof}[Proof of Lemma \ref{SIV_ch}] 
Suppose that $\eta\in (C^{3}(\overline{\Omega}))^{N}$
where $\eta\cdot \nu=0$ on $\p\Omega$. 
 Then for $\tau>0$ small, we 
let $\Psi: \Omega\times (-\tau,\tau)\rightarrow\Omega$ denote the unique solution to the following system of ordinary differential equations
\begin{equation}
\frac{\p \Psi}{\p t}(x, t)= \eta (\Psi(x, t)), \quad\Psi(x, 0)= x. \label{ode}
\end{equation}
Then we have the expansion
\begin{equation}\Psi(x, t)= x + t\eta (x)+ \frac{t^2}{2} Z(x) + O(t^3)\quad\mbox{where}\quad
Z:=(\eta\cdot\nabla )\eta.
\label{Psi_exp}
\end{equation}

Letting $\Phi_t(x) := \Psi(x, t)$
we observe that for all $t$ such that $\abs{t}<\tau$, the mapping $x\mapsto\Psi(x,t)$ is a diffeomorphism of $\Omega$ into itself, using the tangency of $\eta$ along the boundary.

From \eqref{innerdefn2} and \eqref{Psi_exp} we have
\[\delta^{2} A(u,\eta,Z) = \left.\frac{d^2}{dt^2}\right\rvert_{t=0} A(u\circ\Phi_t^{-1})=  \left.\frac{d^2}{dt^2}\right\rvert_{t=0} A(u_t),\quad\text{and } \quad\delta A(u,\eta)= \left.\frac{d}{dt}\right\rvert_{t=0} A(u_t)\]
where
$u_t(y) := u(\Phi_t^{-1}(y)).$
By the change of variables $y=\Phi_{t}(x)$ and using $\Phi_t^{-1}(\Omega)=\Omega$, we have 
\begin{eqnarray}
 A(u_{t}) &=& \int_{\Phi_t^{-1}(\Omega)} F(u(x),\nabla u\cdot\nabla \Phi_{t}^{-1}(\Phi_{t}(x)) \abs{\text{det}\nabla \Phi_{t}(x)}dx\nonumber
\\
&=& \int_{\Omega} F(u(x),\nabla u\cdot\nabla \Phi_{t}^{-1}(\Phi_{t}(x)) \abs{\text{det}\nabla \Phi_{t}(x)}dx. 
\label{rewrite_E}
\end{eqnarray}
We need to expand the right-hand side of the above formula up to the second power in $t$. 
Note that
\begin{equation*}
 \nabla\Phi_{t}^{-1}(\Phi_{t}(x)) = [I + t\nabla\eta (x) +\frac{t^2}{2}\nabla Z(x) + O(t^3)]^{-1} = I - t\nabla\eta -\frac{t^2}{2}\nabla Z(x)+ t^{2}(\nabla\eta)^2 + O(t^3),
\end{equation*}
hence
\begin{equation}
 \nabla u\cdot\nabla \Phi_{t}^{-1}(\Phi_{t}(x)) = \nabla u - t\nabla u\cdot\nabla\eta -\frac{t^2}{2}\nabla u\cdot\nabla Z(x)+ t^{2}(\nabla\eta)^2\cdot\nabla u + O(t^3).
\label{u1_expand}
\end{equation}
We then use the following identity for matrices $A$ and $B$
\begin{equation*}
 \text{det}(I + tA + \frac{t^{2}}{2}B) = 1 + t\,\text{trace}(A) + \frac{t^2}{2}[\text{trace}(B) + (\text{trace}(A))^2 - \text{trace}(A^2)] + O(t^3).
\end{equation*}
Therefore, since for $\abs{t}$ sufficiently small, $\text{det}\nabla\Phi_{t}(x)>0$ and we find
\begin{multline}
\abs{\text{det}\nabla\Phi_{t}(x)}=\text{det}\nabla\Phi_{t}(x) =\text{det} (I + t\nabla \eta (x) +\frac{t^2}{2}\nabla Z)\\= 1 +  t\,\text{div} \,\eta + \frac{t^2}{2}[ \text{div}\,Z + (\text{div}\eta)^2 - \text{trace}((\nabla\eta)^2)] + O(t^3).
\label{det_expand}
\end{multline}
Plugging (\ref{u1_expand}) and (\ref{det_expand}) into (\ref{rewrite_E}), we find that
\begin{equation}
\label{inner_formu}
\delta A(u,\eta)= \left.\frac{d}{dt}\right\rvert_{t=0} \int_{\Omega} \hat F(x, t) dx~\text{and } \delta^2A(u,\eta,Z)= \left.\frac{d^2}{dt^2}\right\rvert_{t=0} \int_{\Omega} \hat F(x, t) dx
\end{equation}
where
$$\hat F(x,t)=F(u, \nabla u - t\nabla u\cdot\nabla\eta - t^2 Y) (1 + t\,{\rm div}\, \eta + \frac{t^2}{2}X).$$
Here $X$ and $Y$ are defined as in (\ref{XY_eq}).

We compute
\begin{equation}
\begin{split}
\label{Fhat_d}
\frac{\partial}{\partial t} \hat F(x,t) =-F_{p_i}(u, \nabla u - t\nabla u\cdot\nabla\eta - t^2 Y) (\frac{\partial}{x_i}\eta^j u_j + 2t Y^i) \\ + F\,{\rm div}\, \eta + \left( \frac{d}{dt}F(u, \nabla u - t\nabla u\cdot\nabla\eta - t^2 Y)\right) t \,{\rm div}\, \eta
+ t FX + \frac{t^2}{2} \frac{d}{dt} (FX).
\end{split}
\end{equation}
The formula for the first inner variation $\delta A(u,\eta)$ easily follows from (\ref{inner_formu}) and (\ref{Fhat_d}). For the second inner variation, we note that
\begin{multline}
\label{Fhat_dd}
\left.\frac{\p^2}{\p t^2}\right\rvert_{t=0} \hat F(x,t)= F_{p_i p_k} (\frac{\partial}{x_i}\eta^j u_j)(\frac{\partial}{x_k}\eta^l u_l) - 2F_{p_i} Y^i \\+ 2 \frac{d}{dt}F(u, \nabla u - t\nabla u\cdot\nabla\eta - t^2 Y)\,{\rm div}\,\eta\mid_{t=0} 
+ FX\\
= FX - 2 (\nabla_{\bf p} F, \nabla u\cdot \nabla\eta)\,{\rm div}\, \eta - 2 (\nabla_{\bf p} F, Y) + F_{p_i p_j}(\nabla u\cdot\nabla \eta)^{i}(\nabla u\cdot\nabla \eta)^{j}.
\end{multline}
Therefore, from (\ref{inner_formu}) and (\ref{Fhat_dd}), we find that
the second inner variation $\delta^{2} A(u,\eta, Z)$ is given by
\begin{multline*}\delta^{2} A(u,\eta, Z)=
\left.\frac{d^2}{dt^2}\right\rvert_{t=0} \int_{\Omega} \hat F(x, t) dx=  \int_{\Omega} \left.\frac{\p^2}{\p t^2}\right\rvert_{t=0} \hat F(x,t) dx
\\= \int_{\Omega}\left\{ FX - 2 (\nabla_{\bf p} F, \nabla u\cdot \nabla\eta) \,{\rm div}\,\eta - 2 (\nabla_{\bf p} F, Y) + F_{p_i p_j}(\nabla u\cdot\nabla \eta)^{i}(\nabla u\cdot\nabla \eta)^{j}\right\} dx.
\end{multline*}
\end{proof}

\section{Passage to the limit in the inner variations of the Allen-Cahn functional}
\label{AC_sec}
In this section we will apply the formulae established in the previous section to the case of the Allen-Cahn or Modica-Mortola sequence of functionals 
\begin{equation} E_{\varepsilon}(u)=\int_{\Omega}\left(\frac{\varepsilon \abs{\nabla u}^2}{2} +\frac{(1-u^2)^2}{2\varepsilon}\right) dx,\label{MM}\end{equation} 
for $\e>0$, where $u: \Omega\subset\RR^N\rightarrow \RR$, $N\geq 2$. Thus, we specialize to the case where $F(z,{\bf p})=\frac{\e}{2}\abs{{\bf p}}^2+\frac{(1-z^2)^2}{2\varepsilon}$ in \eqref{basic}. These functionals, which in particular arise in the
theory of phase transitions \cite{AC}, are known to $\Gamma$-converge in $L^1(\Omega)$ to a multiple of the perimeter functional
$E$ defined by
\begin{equation*}
E(u_0)=\left\{
 \begin{alignedat}{1}
   \frac{1}{2}\int_{\Omega}|\nabla u_0| ~&~ \text{if} ~u_0\in BV (\Omega, \{1, -1\}), \\\
\infty~&~ \text{otherwise},
 \end{alignedat} 
  \right.
  \end{equation*}
  (\cite{MM}).
  More precisely, $E_\e$ $\Gamma$-converges in $L^1(\Omega)$ to $\frac{4}{3}E.$
  
  For a function $u_0$ of bounded variation taking values $\pm 1$, i.e. $u_0\in BV (\Omega, \{1, -1\})$,  $|\nabla u_0|$ denotes the total variation of the vector-valued measure $\nabla u_0$ (see \cite{Giusti}), and $\Gamma:= \partial\{x\in \Omega: u_0(x)=1\}\cap \Omega$ denotes the interface separating 
the $\pm1$ phases of $u_0$. If $\Gamma$ is sufficiently regular, say $C^1$, then $E(u_0)=\mathcal{H}^{N-1}(\Gamma)$ and hence we identify 
\begin{equation}
\label{E_defn}
E(u_0)\equiv E(\Gamma)=\mathcal{H}^{N-1}(\Gamma)
\end{equation} where $\mathcal{H}^{N-1}$ denotes $(N-1)$-dimensional Hausdorff measure. Throughout, we will denote by ${\bf n} = (n_{1},\cdots,n_{N})$ the unit normal to $\Gamma$ pointing out of the region $\{x\in\Omega:\,u_0(x)=1\}.$ 

Though we will not use the specific properties of $\Gamma$-convergence in this article, we recall that this convergence of $E_\e$ to $\frac{4}{3}E$ consists of two conditions: a liminf inequality and the existence of a recovery sequence.
For reader's convenience and for later reference, we give the definition below.
\begin{definition} [$\Gamma$-convergence]
\label{G_defn}
We say that a sequence of functionals $E_\e$ $\Gamma$-converges in the $L^1(\Omega)$ topology to the functional $\frac{4}{3}E$ if 
for any $u\in L^1(\Omega)$ one has the following two conditions:
\begin{myindentpar}{1cm}
(i) (Liminf inequality) If a sequence $\{v_\e\}$ converges to
$u$ in $L^1(\Omega)$, then
$$\liminf_{\e\to 0}E_\e(v_\e)\geq \frac{4}{3}E(u), $$
 (ii) (Existence of a recovery sequence) There exists a sequence $\{w_\e\}\subset L^1(\Omega)$ converging to $u$ such that $$\lim_{\e\to 0}E_\e(w_\e)=\frac{4}{3}E(u).$$
  \end{myindentpar}
  \end{definition}
This convergence, when accompanied by a compactness condition on energy-bounded sequences, guarantees that global minimality passes to the limit. In this article, however, we will be more concerned with the passage of stability in the limit $\e\to 0.$

The first variation of $E$, defined by (\ref{E_defn}), at $\Gamma$ with respect to a smooth velocity vector field $\eta$ is given by
\begin{eqnarray}
\label{FVE}
 \delta E(\Gamma,\eta)&: =& \left.\frac{d}{dt}\right\rvert_{t=0}  \mathcal{H}^{N-1}(\Phi_t(\Gamma))= 
 \int_{\Gamma}\text{div}^{\Gamma}\eta\mathcal{H}^{N-1},
\end{eqnarray}
and 
the second variation of $E$ at $\Gamma$ with respect to smooth velocity and acceleration vector fields $\eta$ and $\zeta$ is given by
\begin{eqnarray}
\label{SVE}
 \delta^{2}E(\Gamma,\eta,\zeta)&: =& \left.\frac{d^2}{dt^2}\right\rvert_{t=0}  \mathcal{H}^{N-1}(\Phi_t(\Gamma))\nonumber\\ &=& 
 \int_{\Gamma}\left\{ \text{div}^{\Gamma}\zeta + (\text{div}^{\Gamma}\eta)^2 + \sum_{i=1}^{N-1}\abs{(D_{\tau_{i}}\eta)^{\perp}}^2 -
 \sum_{i,j=1}^{N-1}(\tau_{i}\cdot D_{\tau_{j}}\eta)(\tau_{j}\cdot D_{\tau_{i}}\eta)\right\}d\mathcal{H}^{N-1};
\end{eqnarray}
see \cite[Chapter 2]{Simon}.
Here 
$\Phi_t$ is 
given by (\ref{map-deform2}), $\text{div}^{\Gamma}\varphi$ denotes the tangential divergence of $\varphi$ on $\Gamma$, and for each point $x\in\Gamma$, $\{\tau_{1}(x),\cdots,\tau_{N-1}(x)\}$ is any orthonormal basis for the tangent space $T_{x}(\Gamma)$. Further, for each $\tau\in T_{x}(\Gamma)$, $D_{\tau} \eta$ is the directional derivative and the normal part of $D_{\tau_{i}} \eta$ is denoted by 
 $(D_{\tau_{i}} \eta)^{\perp} := D_{\tau_{i}} \eta -\sum_{j=1}^{N-1}(\tau_{j}\cdot D_{\tau_{i}} \eta)\tau_{j}. $ We point that out for a hypersurface, there are no distinct notions of first or second {\it inner} variation so while we chose the notation $ \delta E(\Gamma,\eta)$ and $\delta^{2}E(\Gamma,\eta,\zeta)$ we could just as well have used $ dE(\Gamma,\eta)$ and $d^{2}E(\Gamma,\eta,\zeta)$.

For later use, we also record the following (see  \cite[formula (12.39)]{SZ2}):
\begin{thm} [Second variation of the area functional \cite{SZ2}] 
\label{PI_ineq}
Suppose that $\Gamma\subset\Omega$ is a smooth hypersurface with mean curvature $H$. Suppose further that, $\overline{\Gamma}$ is $C^2$ and  $\partial\Gamma$ meets $\partial\Omega$ orthogonally.
Then for any smooth vector field $\eta:\overline{\Omega}\to \R^N$ that is tangent to $\partial\Omega$ with $\eta=\xi\,{\bf n}$ and
 $({\bf n},{\bf n}\cdot\nabla\eta) =0$  on $\Gamma$ for some smooth $\xi:\Gamma\to\R$, and for $Z:= (\eta \cdot \nabla) \eta$, we have
\begin{multline}
\delta^{2}E(\Gamma,\eta, Z)  =\delta^{2}E(\Gamma,\xi):  = \int_{\Gamma} \left(|\nabla_{\Gamma}\xi|^2 + (n-1)^2H^2\xi^2 -|A_{\Gamma}|^2\xi^2\right) d\mathcal{H}^{N-1} \\- \int_{\partial\Gamma\cap\partial\Omega} A_{\partial\Omega}({\bf n}, {\bf n})\xi^2 d\mathcal{H}^{N-2}.\label{sziden}
\end{multline}
Here $A_{\Gamma}$ and $A_{\partial\Omega}$ denote the second fundamental form of $\Gamma$ and $\p\Omega$ respectively. 
\end{thm}
\begin{rem}
The derivation of \cite[formula (12.39)]{SZ2} uses the stability of $\Gamma$ only in order to assert the necessary regularity to carry out the calculation. Here, as we do throughout the article, we assume smoothness of $\overline{\Gamma}$ so a stability assumption is not needed. 
\end{rem}

In a previous paper \cite{Le}, the first author studied the relationship between the second inner variations of $\{E_\e\}$ and the second variation of the $\Gamma$-limit, $\frac{4}{3} E(u_0)$.
While the first inner variations of $E_\e$ converge to the first variation of $E_0$, it was shown in \cite{Le} that an extra positive discrepancy term emerges in the limit of the second inner variation. More precisely, if $u_\e\in C^2(\Omega), u_{\e}\rightarrow u_0\in BV (\Omega, \{1, -1\})$ with a $C^2$ interface $\Gamma$ and $\lim_{\e\rightarrow 0} E_{\e}(u_{\e})= \frac{4}{3}E (u_0)$, then for all smooth vector fields $\eta,\zeta\in (C_{c}^{1}(\Omega))^{N}$, it was found in \cite[Theorem 1.1]{Le} that
\begin{equation*}
\lim_{\varepsilon\rightarrow 0}\delta^{2} E_{\varepsilon}(u_{\varepsilon},\eta,\zeta) = \frac{4}{3}\left\{
 \delta^{2}E(\Gamma,\eta,\zeta) + \int_{\Gamma} ({\bf n},{\bf n}\cdot\nabla\eta)^2d\mathcal{H}^{N-1}\right\}.\end{equation*}

With the aim of studying the asymptotic behavior of Allen-Cahn critical points and linearizations subject the natural Neumann boundary conditions, we now establish the same type of result {\it without} the assumption of compact support on the vector fields $\eta$ and $\zeta$:
 \begin{thm}[Limits of the inner variations of the Allen-Cahn functional]
 \label{thm-AC2}
Let $\{u_{\e}\}\subset C^3 (\overline{\Omega})$ be a sequence of functions that converges in $L^1(\Omega)$ to a function $u_0\in BV (\Omega, \{1, -1\})$ with an interface $\Gamma=\partial\{u_0=1\}\cap\Omega$ having the property that $\overline{\Gamma}$ is $C^2$. Assume that
 $\lim_{\varepsilon \rightarrow 0} E_{\varepsilon}(u_{\varepsilon}) = 
\frac{4}{3}E(\Gamma).$
Then, for all smooth vector fields $\eta\in (C^{3}(\overline{\Omega}))^{N}$ with $\eta\cdot \nu=0$ on $\partial\Omega$ and for $Z:= (\eta \cdot\nabla) \eta$, we 
have
\begin{equation*}\lim_{\varepsilon\rightarrow 0}\delta E_{\varepsilon}(u_{\varepsilon}, \eta) = \frac{4}{3}\delta E(\Gamma,\eta) 
\end{equation*} 
and
\begin{equation*}\lim_{\varepsilon\rightarrow 0}\delta^{2} E_{\varepsilon}(u_{\varepsilon}, \eta, Z) = \frac{4}{3}\left\{\delta^{2}E(\Gamma,\eta, Z) + \int_{\Gamma} ({\bf n},{\bf n}\cdot\nabla\eta)^2 d\mathcal{H}^{N-1}\right\}.
\end{equation*} 
\end{thm}
\begin{rem}
(i) One important point in Theorem \ref{thm-AC2} is that $u_\e$ is not assumed to necessarily be a critical point of $E_{\e}$. We will find ourselves in need of the formula in this situation in Section 6.\\
(ii) In the convergence result for the second inner variations $\delta^{2} E_{\varepsilon}(u_{\varepsilon}, \eta, Z)$ in Theorem \ref{thm-AC2}, it would be very interesting to relax the hypothesis 
 $\lim_{\varepsilon \rightarrow 0} E_{\varepsilon}(u_{\varepsilon}) = 
\frac{4}{3}E(\Gamma)$ (which amounts to assuming multiplicity 1 convergence of the nodal sets of $u_\e$) to just a uniform bound $E_\e(u_\e)\leq C$ on the energies $E_\e(u_\e)$ as done by 
Gaspar \cite[Proposition 3.3]{Ga} for compactly supported vector fields $\eta$ (and hence $Z$). Gaspar's elegant observation (see \cite[Proposition 2.2]{Ga}) is that, under the energy bound $E_\e(u_\e)\leq C$ and 
the vanishing of the discrepancy measures $\xi_{\e}:=\left(\e |\nabla u_\e|^2 - \frac{(1-u_\e^2)^2}{\e}\right)$ in the interior of $\Omega$, the second inner variations $\delta^2 E_\e(u_\e, \cdot,\cdot)$ are continuous under varifold convergence in the interior of the Euclidean domain $\Omega$. 
The application of this continuity result to $\delta^{2} E_{\varepsilon}(u_{\varepsilon}, \eta, Z)$ requires
the vector fields $\eta$ and $Z:=(\eta\cdot\nabla)\eta$ be compactly supported in $\Omega$, which is not the case in our present setting of Theorem \ref{thm-AC2}, however. On the other hand, adapting the analysis of
\cite{Ga} for general vector fields $\eta$ requires the vanishing of the discrepancy measures $\xi_{\e}$ up to the boundary of $\Omega$ in the limit of $\e\rightarrow 0$.
To the best of the authors' knowledge, the most general setting for the validity of this result is the work of Mizuno-Tonegawa \cite{MT} where the authors require that $u_\e$ is a critical point of $E_\e$, uniformly bounded in $\e$ and that the domain $\Omega$ is strictly convex. Recently, Kagaya \cite{Ka} relaxes 
the strict convexity of the Euclidean domain $\Omega$ for certain classes of critical points of $E_{\e}$. 
Note that  our present setting of Theorem \ref{thm-AC2} (see also item (i) above) does not fulfill these requirements in general. We briefly sketch a generalization of Theorem \ref{thm-AC2} to the setting of \cite{MT} in Theorem \ref{thm-ACc}.
\end{rem}

The rest of this section is devoted to proving Theorem \ref{thm-AC2} and its slight generalization, Theorem \ref{thm-ACc}.
Let
\begin{equation}
\label{Phi_a}
\Phi (a): =  \int_{0}^{a} |s^2-1| ds.
\end{equation}
 We next recall the following results from
\cite[Lemmas 3.1 and 3.2]{Le2} applied to the double well potential $(1-u^2)^2$ that are crucial to proving Theorem \ref{thm-AC2}. 
\begin{lemma}
\label{equi_lem}
Under the assumptions of Theorem \ref{thm-AC2}, we have the following convergences:
\begin{myindentpar}{1cm}
\begin{equation}
\label{ener1bis}
\lim_{\e\rightarrow 0}\int_{\Omega}|\nabla \Phi (u_\e)| dx= \int_{\Omega}|\nabla \Phi (u_0)|dx,
\end{equation}
\begin{equation}
\label{equi-ab}
\lim_{\e\rightarrow 0} \int_{\Omega} |\e |\nabla u_\e|^2 - \frac{(1-u_\e^2)^2}{\e}| dx =0,
\end{equation}
\begin{equation}\label{equi-abphi}\lim_{\e\rightarrow 0}\int_{\Omega}|\e |\nabla u_\e|^2- |\nabla \Phi (u_\e)|| dx=0.
\end{equation}
\end{myindentpar}
We also  have the following convergence:
\begin{equation}\label{L1_con}\Phi (u_\e)\rightarrow \Phi (u_0)~\text{in } L^1 (\Omega)
\end{equation}
and thus, in the sense of Radon measures, we have the convergence:
\[
\nabla \Phi(u_\e)\rightharpoonup \nabla \Phi (u_0)= \frac{4}{3}  {\bf n} \,d\mathcal{H}^{N-1}\lfloor\Gamma\;\mbox{as}\;\e\to 0.
\]
\end{lemma}
\begin{rem}
In the special case where $u_\e$ is a minimizer of $E_\e$, the above lemma was proved by Luckhaus and Modica; see \cite[Proposition 1, Lemmas 1 and 2]{LM}.
Equation (\ref{L1_con}) in Lemma \ref{equi_lem} was used in \cite{Le2} without proof. Its proof is based on  a truncation argument as in the proof of (1.11) in \cite{Stern1}. For completeness, we include it below.
\end{rem}
\begin{proof} [Proof of equation (\ref{L1_con}) in Lemma \ref{equi_lem}]
Let us define
$$u^{\ast}_\e = u_\e~\text{on } \{-1\leq u_\e\leq 1\} ~\text{ and }u_\e^{\ast} = \text{sign} (u_\e) ~\text{on }\{|u_\e|>1\}.$$
First, note that $u_\e\rightarrow u_0$ in $L^1 (\Omega)$ implies that $u_\e^{\ast}\rightarrow u_0$ in $L^1(\Omega)$. Moreover
$\Phi (u^{\ast}_\e)\rightarrow \Phi (u_0) $ in $L^1 (\Omega)$.
It suffices to show that
$\Phi (u^{\ast}_\e)\rightarrow \Phi (u_\e)~\text{in } L^1 (\Omega).$
Since
$$\int_{\Omega} |\Phi(u_\e)-\Phi(u^{\ast}_\e)| dx= \int_{\{|u_\e|>1\}} |\Phi(u_\e)-\Phi(\text{sign} (u_\e))| dx,$$
by symmetry, it suffices to show that 
\begin{equation}\lim_{\e\rightarrow 0} \int_{\{u_\e>1\}} |\Phi(u_\e)-\Phi(1)|dx=0.
\label{above1_eq}
\end{equation}
From the construction of $u_\e^{\ast}$, we have
\begin{eqnarray*}
E_{\e} (u_\e) = \int_{\Omega} \left( \frac{\e |\nabla u_\e|^2}{2} + \frac{(u_\e^2-1)^2}{2\e} \right)dx
= E_\e (u^\ast_\e) + \int_{\{|u_\e|>1\}} \left( \frac{\e |\nabla u_\e|^2}{2} + \frac{(u_\e^2-1)^2}{2\e} \right)dx.
\end{eqnarray*}
By the liminf inequality in the $\Gamma$-convergence of $E_\e$ to $\frac{4}{3} E$ (see Definition \ref{G_defn}), we have from $u_\e^{\ast}\rightarrow u_0$ in $L^1(\Omega)$ that
$$\liminf_{\e\rightarrow 0} E_{\e}(u^\ast_\e)\geq \frac{4}{3} E(u_0)=\frac{4}{3} E(\Gamma).$$
Because
$\lim_{\e\rightarrow 0} E_{\e}(u_\e)= \frac{4}{3} E(\Gamma),$
we find that
\begin{equation}\lim_{\e\rightarrow 0} \int_{\{|u_\e|>1\}} \left( \frac{\e |\nabla u_\e|^2}{2} + \frac{(u_\e^2-1)^2}{2\e} \right) dx=0.
\label{lim01}
\end{equation}
When $u_\e>1$, we have from the definition of $\Phi$ in (\ref{Phi_a}) that
$\Phi(u_\e)-\Phi (1) = (u_\e-1)^2 (u_\e+2)/3.$
Thus, using (\ref{lim01}), we obtain
\begin{equation*}
 \int_{\{u_\e>1\}} |\Phi(u_\e)-\Phi(1)| dx\leq   \int_{\{u_\e>1\}} (u_\e-1)^2 (u_\e+2)dx\leq  \int_{\{u_\e>1\}} (u_\e^2-1)^2 dx\rightarrow 0~\text{when } \e\rightarrow 0.
\end{equation*}
The proof of (\ref{above1_eq}) is complete.
\end{proof}

Before recalling a theorem of Reshetnyak, we introduce some notation.
Let $ [C_0(\Omega)]^m$ be the space of $\RR^m$-valued continuous functions with compact support in $\Omega$. Let $[M_b(\Omega)]^m$ be the space of $\RR^m$-valued measures on $\Omega$ with finite total mass. Given $\mu\in [M_b(\Omega)]^m$, we write $|\mu|$ for the total variation of $\mu$ and $\frac{d\mu}{d|\mu|}$ for the Radon-Nikodym derivative of $\mu$ with respect to $|\mu|$.   Given $\mu_n,\mu\in [M_b(\Omega)]^m$, we say that $\mu_n$ converges to $\mu$ in the sense of Radon measures if for all $\varphi \in [C_0(\Omega)]^m$, we have
$$\lim_{n\rightarrow\infty} \int_{\Omega}\varphi\cdot d\mu_n=\int_{\Omega}\varphi\cdot d\mu.$$
We now recall 
a theorem of Reshetnyak \cite{Res} concerning continuity of functionals with respect to Radon convergence of measures. Its equivalent form that we write down below is taken from Spector \cite[Theorem 1.3]{Sp}. 
\begin{thm}
[Reshetnyak's continuity theorem]
\label{Sp_thm}
Let $\Omega\subset\RR^N$ be open, $\mu_n,\mu\in [M_b(\Omega)]^m$ be such that
$\mu_n$ converges to $\mu$ in the sense of Radon measures and $|\mu_n|(\Omega)\rightarrow |\mu|(\Omega)$. Then
\begin{equation*}
\lim_{n\rightarrow\infty} \int_{\Omega} f\left(x, \frac{d\mu_n}{d|\mu_n|}(x)\right) d|\mu_n|= \int_{\Omega} f\left(x, \frac{d\mu}{d|\mu|}(x)\right) d|\mu|
\end{equation*}
for every continuous and bounded function $f:\Omega\times \mathcal{S}^{m-1}\rightarrow \RR$ where $\mathcal{S}^{m-1}:=\{x\in\R^m: |x|=1\}$.
\end{thm}
We emphasize that in Theorem \ref{Sp_thm}, $f$ is not required to be compactly supported in $\Omega$. This is crucial to applications in our paper.

The following lemma provides a key ingredient in the proof of Theorem \ref{thm-AC2}. It allows us to pass to the limit in certain quadratic expressions involving $\nabla u_\e$.
\begin{lemma}
\label{Res_Sp} Under the assumptions of Theorem \ref{thm-AC2},
for all $\varphi\in C(\overline{\Omega})$, we have
\begin{equation}
\int_{\Omega}\e \nabla u_{\e}\otimes \nabla u_{\e} \varphi dx\rightarrow \frac{4}{3} \int_{\Gamma}{\bf n}\otimes {\bf n}\,\varphi\, d\mathcal{H}^{N-1}.
\end{equation}
\end{lemma}
\begin{proof}[Proof of Lemma \ref{Res_Sp}]
The proof is a simple application of Theorem \ref{Sp_thm} using Lemma \ref{equi_lem}. Let $\Phi$ be as in (\ref{Phi_a}).
We have
$$\e \nabla u_{\e}\otimes \nabla u_{\e} =\frac{\nabla u_{\e}}{|\nabla u_{\e}|}\otimes \frac{\nabla u_{\e}}{|\nabla u_{\e}|} \e |\nabla u_\e|^2= \frac{\nabla \Phi(u_{\e})}{|\nabla \Phi (u_{\e})|}\otimes \frac{\nabla \Phi(u_{\e})}{|\nabla \Phi(u_{\e})|} 
\e |\nabla u_\e|^2.$$
From equation (\ref{equi-abphi}) in Lemma \ref{equi_lem}, we find that for any $\varphi\in C(\overline{\Omega}),$
\begin{equation*}
\lim_{\e\rightarrow 0}\int_{\Omega}\frac{\nabla \Phi(u_{\e})}{|\nabla \Phi (u_{\e})|}\otimes \frac{\nabla \Phi(u_{\e})}{|\nabla \Phi(u_{\e})|} \e |\nabla u_\e|^2 \varphi dx= \lim_{\e\rightarrow 0} \int_{\Omega}\frac{\nabla \Phi(u_{\e})}{|\nabla \Phi (u_{\e})|}\otimes \frac{\nabla \Phi(u_{\e})}{|\nabla \Phi(u_{\e})|} |\nabla \Phi (u_{\e})| \varphi dx.
\end{equation*}
Applying Theorem \ref{Sp_thm} to $\nabla \Phi(u_\e)$ and $\nabla \Phi (u_0)$ with $f(x, {\bf p}) =\left( {\bf p}\otimes {\bf p}\right)\,\varphi(x)$, we find
\begin{equation*}
\lim_{\e\rightarrow 0} \int_{\Omega}\frac{\nabla \Phi(u_{\e})}{|\nabla \Phi (u_{\e})|}\otimes \frac{\nabla \Phi(u_{\e})}{|\nabla \Phi(u_{\e})|} |\nabla \Phi (u_{\e})| \varphi dx
= \int_{\Gamma} \frac{4}{3} {\bf n}\otimes {\bf n}\,\varphi\, d\mathcal{H}^{N-1}.
\end{equation*}
\end{proof}
We can now present:
\begin{proof}[Proof of Theorem \ref{thm-AC2}]
When
$\eta\cdot \nu=0$ on $\partial\Omega$ and $Z:=  (\eta\cdot\nabla) \eta$, Lemma \ref{SIV_ch} applied to $E_{\e}$ gives
\begin{eqnarray}
\delta E_{\varepsilon}(u_{\varepsilon},\eta) =\int_{\Omega}  \left[\left( \frac{\e |\nabla u_\e|^2}{2} + \frac{(u_\e^2-1)^2}{2\e}\right) \div\eta - \e(\nabla u^{\varepsilon},\nabla u^{\varepsilon}\cdot\nabla\eta)\right]dx\label{star}
\end{eqnarray}
and
\begin{multline}
\delta^{2} E_{\varepsilon}(u_{\varepsilon},\eta, Z) =\int_{\Omega} \left\{ \left( \frac{\e |\nabla u_\e|^2}{2} + \frac{(u_\e^2-1)^2}{2\e}\right) \left(\text{div} Z + (\text{div}\eta)^2 -
\text{trace}((\nabla \eta)^2)\right) \right\}dx\\ 
-2\int_{\Omega}\e(\nabla u^{\varepsilon},\nabla u^{\varepsilon}\cdot\nabla\eta)\text{div}\eta dx
-2\int_{\Omega}\left(\e\nabla u^{\varepsilon}, \frac{1}{2}\nabla u^{\varepsilon}\cdot\nabla Z - (\nabla\eta)^2\cdot \nabla u^\e\right) dx
+ \int_{\Omega}\e\abs{\nabla u^{\varepsilon}\cdot\nabla\eta}^2 dx.
\label{svep-p}\end{multline}
By letting $\e\rightarrow 0$ and using Lemmas \ref{equi_lem} and \ref{Res_Sp} together with (\ref{FVE}), we find that
\begin{equation*}
\lim_{\e\rightarrow 0} \delta E_{\varepsilon}(u_{\varepsilon},\eta)= \frac{4}{3}\int_{\Gamma}( \div \eta- ({\bf n}, {\bf n}\cdot\nabla\eta)) d\mathcal{H}^{N-1} = \frac{4}{3} \delta E(\Gamma,
\eta).
\end{equation*}
Let us now analyze $\delta^{2} E_{\varepsilon}(u_{\varepsilon},\eta, Z) $.
Using equation (\ref{equi-ab}) in Lemma \ref{equi_lem} together with Lemma \ref{Res_Sp}, we find that
\begin{multline*}
\lim_{\e\rightarrow 0 }\int_{\Omega} \left( \frac{\e |\nabla u_\e|^2}{2} + \frac{(u_\e^2-1)^2}{2\e} \right)\left(\text{div}Z + (\text{div}\eta)^2 -
\text{trace}((\nabla \eta)^2)\right) 
\\ = \int_{\Omega} \e |\nabla u_\e|^2 \left(\text{div} Z + (\text{div}\eta)^2 -
\text{trace}((\nabla \eta)^2)\right)
= \frac{4}{3} \int_{\Gamma} \left(\text{div} Z + (\text{div}\eta)^2 -
\text{trace}((\nabla \eta)^2)\right)d\mathcal{H}^{N-1}.
\end{multline*}
By letting $\e\rightarrow 0$ and using Lemma \ref{Res_Sp}, we obtain
\begin{multline}\lim_{\varepsilon\rightarrow 0}\delta^{2} E_{\varepsilon}(u_{\varepsilon}, \eta, Z) =
\frac{4}{3}\int_{\Gamma}\left\{\text{div} Z + (\text{div}\eta)^2 -
\text{trace}((\nabla \eta)^2- 2({\bf n},{\bf n}\cdot\nabla\eta) \text{div} \eta \right\}d\mathcal{H}^{N-1}\\ - 
\frac{8}{3} \int_{\Gamma}({\bf n}, \frac{1}{2}{\bf n}\cdot\nabla Z -{\bf n} \cdot (\nabla\eta)^2 )d\mathcal{H}^{N-1} + 
\frac{4}{3}\int_{\Gamma}|{\bf n}\cdot \nabla \eta|^2 d\mathcal{H}^{N-1}.
\label{SVpp}
\end{multline} 
As in the proof of Theorem 1.1 in \cite{Le} (see (2.8) there), we find that
\begin{multline*}
 \text{div} Z + (\text{div}\eta)^2 -
\text{trace}((\nabla \eta)^2- 2({\bf n},{\bf n}\cdot\nabla\eta) \text{div} \eta -
2 ({\bf n}, \frac{1}{2}{\bf n}\cdot\nabla Z -{\bf n} \cdot (\nabla\eta)^2 ) + 
|{\bf n}\cdot \nabla \eta|^2 \\=
\text{div}^{\Gamma} Z + (\text{div}^{\Gamma}\eta)^2 + \sum_{i=1}^{N-1}\abs{(D_{\tau_{i}}\eta)^{\perp}}^2 -
 \sum_{i,j=1}^{N-1}(\tau_{i}\cdot D_{\tau_{j}}\eta)(\tau_{j}\cdot D_{\tau_{i}}\eta) +  ({\bf n},{\bf n}\cdot\nabla\eta)^2.
\end{multline*}
In light of  (\ref{SVE}), we find that the right hand side of (\ref{SVpp}) is equal to 
\[
\frac{4}{3}\left\{\delta^{2}E(\Gamma,\eta, Z) + \int_{\Gamma} ({\bf n},{\bf n}\cdot\nabla\eta)^2d\mathcal{H}^{N-1}\right\} .
\]
Therefore, we obtain the desired formula for $\lim_{\varepsilon\rightarrow 0}\delta^{2} E_{\varepsilon}(u_{\varepsilon}, \eta, Z) $ as stated in the theorem.
\end{proof}

For the remainder of this section, we briefly sketch a generalization of Theorem \ref{thm-AC2} to the special setting of Allen-Cahn critical points with a Neumann boundary condition on strictly convex domains.
 \begin{thm}[Limits of the inner variations of the Allen-Cahn functional with a uniform energy bound on strictly convex domains]
 \label{thm-ACc}
 Assume that $\Omega$ is an open, smooth, bounded and strictly convex domain in $\R^N~(N\geq 2)$.
Let $\{u_{\e_j}\}\subset C^3 (\overline{\Omega})$ be a sequence of critical points of the Allen-Cahn functionals $E_{\e_j}$ that converges in $L^1(\Omega)$ to a function $u_0\in BV (\Omega, \{1, -1\})$ with an interface $\Gamma=\partial\{u_0=1\}\cap\Omega$ having the property that $\overline{\Gamma}$ is $C^2$. 
Assume that there is a positive constant $C$ such that
 $\|u_{\e_j}\|_{L^{\infty}(\Omega)} + E_{\e_j}(u_{\e_j}) \leq C
$ for all $j$. Let $\Gamma_1, \cdots, \Gamma_K$ be connected components of $\Gamma$.
Then, 
\begin{myindentpar}{1cm}
(i) there are positive integers $m_1, \cdots, m_K$ such that
$$\lim_{j\rightarrow \infty} E_{\e_j}(u_{\e_j}) = \frac{4}{3}\sum_{i=1}^K m_i E(\Gamma_i); $$
(ii) for all smooth vector fields $\eta\in (C^{3}(\overline{\Omega}))^{N}$ with $\eta\cdot \nu=0$ on $\partial\Omega$ and for $Z:= (\eta \cdot\nabla) \eta$, we 
have
\begin{equation*}\lim_{j\rightarrow \infty}\delta E_{\e_j}(u_{\e_j}, \eta) = \frac{4}{3}\sum_{i=1}^K m_i \delta E(\Gamma_i,\eta) 
\end{equation*} 
and
\begin{equation*}\lim_{j\rightarrow \infty}\delta^{2} E_{\e_j}(u_{\e_j}, \eta, Z) = \frac{4}{3}\left\{\sum_{i=1}^K m_i \left(\delta^{2}E(\Gamma_i,\eta, Z) +\int_{\Gamma_i} ({\bf n},{\bf n}\cdot\nabla\eta)^2 d\mathcal{H}^{N-1}\right)\right\}.
\end{equation*} 
\end{myindentpar}
\end{thm}
\begin{proof}[Sketch of Proof of Theorem \ref{thm-ACc}] (i) By the criticality of $u_{\e_j}$, $\Gamma$ is a minimal surface.  By the connectedness of each $\Gamma_i$, the conclusion in (i) follows from the Constancy Theorem 
for stationary varifolds \cite[Theorem 41.1]{Simon}; see, for example,  \cite[p. 1854]{Le} or the paragraph following Theorem 2.1 in \cite{Ga}.\\
(ii) From the uniform bound $\|u_{\e_j}\|_{L^{\infty}(\Omega)} + E_{\e_j}(u_{\e_j}) \leq C
$, the criticality of $u_{\e_j}$ which implies that $u_{\e_j}$ satisfies the Neumann boundary condition $\frac{\p u_{\e_j}}{\p \nu}=0$ on $\p\Omega$ and the strict convexity of $\Omega$, we can use \cite[Proposition 6.4]{MT} to conclude the following vanishing property of the discrepancy measure (or equi-partition  of energy)
\begin{equation}
\label{equi-abc}
\lim_{j\rightarrow \infty} \int_{\Omega} |\e_j |\nabla u_{\e_j}|^2 - \frac{(1-u_{\e_j}^2)^2}{\e_j}| dx =0,
\end{equation}
With (\ref{equi-abc}) and (i), we can follow the arguments in the proof of Proposition 2.2 in \cite{Ga} to have the following modified version of Lemma \ref{Res_Sp}:
for all $\varphi\in C(\overline{\Omega})$, we have
\begin{equation}
\label{mKc}
\int_{\Omega}\e \nabla u_{\e_j}\otimes \nabla u_{\e_j} \varphi dx\rightarrow \frac{4}{3}\sum_{i=1}^K m_i \int_{\Gamma_i}{\bf n}\otimes {\bf n}\,\varphi\, d\mathcal{H}^{N-1}.
\end{equation}
Now, using (\ref{equi-abc}) and (\ref{mKc}), instead of Lemmas \ref{equi_lem} and \ref{Res_Sp}, in (\ref{star}) and (\ref{svep-p}) in the proof of Theorem \ref{thm-AC2}, we obtain (ii).
\end{proof}

\section{Applications of Second Variation Convergence for Allen-Cahn} 
We now present two applications of our convergence formula for the second inner variation of the Allen-Cahn functional in Theorem \ref{thm-AC2}. The first, Theorem \ref{ACstab}, concerns the passage of stability from critical points of the Allen-Cahn functional to that of the limiting interface. The second concerns an asymptotic upper bound for the Neumann eigenvalues associated with the linearized Allen-Cahn operator. This is the content of Theorem \ref{eigen_thm}.
\subsection{Stable Critical Points Leading to Stable Interfaces}
An interesting and at times subtle question involves the issue of whether stability of a sequence of critical points passes to the limit within the context of $\Gamma$-convergence. This topic has been looked at from a variety of angles, including \cite{Serfaty} where some conditions related to, but not equivalent to, $\Gamma$-convergence are shown to be sufficient to guarantee stability of the limiting object. Interestingly, the verification of one of the two key sufficient conditions in \cite{Serfaty}
for 2D Ginzburg-Landau vortices uses inner variations; see \cite[equation (3.12)]{Serfaty}.

Within the Allen-Cahn context, the question of whether stability of critical points passes to the limiting interface is addressed in  \cite{Tonegawa}.  Assuming that a sequence of Allen-Cahn critical points $\{u_\e\}$ have non-negative second Gateaux variations with respect to compactly supported variations, and assuming that their energies $E_\e(u_\e)$ are uniformly bounded, Tonegawa identifies a limiting varifold and shows that in an appropriately defined weak sense, it has non-negative generalized second variation; see \cite[Theorem 3]{Tonegawa}. Roughly speaking, stability in this weak sense looks like non-negativity of $\delta^{2}E(\Gamma,\xi)$ given by \eqref{sziden} with the boundary integral absent due to the assumption of compact support on $\xi.$ In a subsequent work, Tonegawa and Wickramasekera \cite{TW} show that 
support of the limiting varifold identified in \cite{Tonegawa} is smooth in dimensions $N\leq 7$ while its singular set (if any) has Hausdorff dimension at most $N-8$ in dimensions
$N>7$. As mentioned in the introduction, the convergence and regularity results in \cite{Tonegawa, TW} rely on an important  interior convergence result for the Allen-Cahn equation from the work of \cite{HT}
and interior regularity results from \cite{W} and we are not aware of boundary analogues of these results.

Here, with stronger assumptions on the regularity of the limiting interface up to the boundary and convergence of energies, we establish a result in this vein which incorporates the boundary term.
\begin{thm}
[Stability of the limiting interface]
\label{ACstab}
Let $\{u_{\e}\}\subset C^3 (\overline{\Omega})$ be a sequence of stable critical points of $E_\e$ given in \eqref{MM} that converges in $L^1(\Omega)$ to a function $u_0\in BV (\Omega, \{1, -1\})$ with an interface $\Gamma:=\partial\{u_0=1\}\cap\Omega$ having the property that $\overline{\Gamma}$ is $C^2$. Assume that
 $\lim_{\varepsilon \rightarrow 0} E_{\varepsilon}(u_{\varepsilon}) = 
\frac{4}{3}E(\Gamma)$ where $E$ is given by (\ref{E_defn}). Then for all smooth $\xi:\overline{\Omega}\to\R$ we have the stability inequality
\[
\int_{\Gamma} \left(|\nabla_{\Gamma}\xi|^2  -|A_{\Gamma}|^2\xi^2\right) d\mathcal{H}^{N-1} - \int_{\partial\Gamma\cap\partial\Omega} A_{\partial\Omega}({\bf n}, {\bf n})\xi^2 d\mathcal{H}^{N-2}\geq 0.
\]
\end{thm}
\begin{rem}
The stability criterion given above for a hypersurface subject to Neumann boundary conditions is derived in \cite[Theorem 2.2]{SZ2}.
\end{rem}
\begin{rem}
Under the assumption of $\Gamma$ being an isolated local minimizer of the $\Gamma$-limit $E$ defined as in (\ref{E_defn}), one can of course construct stable, in fact locally minimizing, critical points $u_\e$ of $E_\e$ using the approach of \cite{KS}. In this case, the above stability inequality for $\Gamma$ holds trivially, since local minimality is a stronger assumption than stability.
\end{rem}

\begin{proof}[Proof of Theorem \ref{ACstab}]
We have assumed that the critical points $u_\e$ of the Allen-Cahn functional $E_\e$ have non-negative second Gateaux variation and so by \eqref{seconde} they also have non-negative second inner variations, that is, for all $\eta,\zeta\in (C^3(\overline{\Omega}))^N$, we have $\delta^{2} E_{\varepsilon}(u_{\varepsilon}, \eta, \zeta)\geq 0.$
By Lemma \ref{ortho_lem} below, $\Gamma$ is a minimal surface and $\partial\Gamma$ meets $\partial\Omega$ orthogonally (if at all). Thus, for any smooth function $\xi:\overline{\Omega}\to\R$, we can choose 
a smooth vector field $\eta$ on $\overline{\Omega}$ such that $\eta=\xi{\bf n}$ on $\Gamma$, $\eta\cdot \nu=0$ on $\partial\Omega$ and such that $({\bf n}, {\bf n}\cdot\nabla \eta)=0$ on $\Gamma$. Then applying Theorems \ref{PI_ineq} and \ref{thm-AC2} with $Z:= (\eta \cdot \nabla) \eta$, we find
\[
0\leq \lim_{\varepsilon\rightarrow 0}\frac{3}{4}\delta^{2} E_{\varepsilon}(u_{\varepsilon}, \eta, Z) = \delta^{2}E(\Gamma,\eta, Z)=\delta^2E(\Gamma,\xi)
\]
for all smooth function $\xi:\overline{\Omega}\to\R$, using \eqref{sziden}. The stability inequality is thus established.
\end{proof}

\begin{lemma}
[Minimality of the limiting interface]
\label{ortho_lem}
Let $\{u_{\e}\}\subset C^3 (\overline{\Omega})$ be a sequence of critical points of $E_\e$ that converges in $L^1(\Omega)$ to a function $u_0\in BV (\Omega, \{1, -1\})$ with an interface $\Gamma=\partial\{u_0=1\}\cap\Omega$
having the property that $\overline{\Gamma}$ is $C^2$. Assume that
 $\lim_{\varepsilon \rightarrow 0} E_{\varepsilon}(u_{\varepsilon}) = 
\frac{4}{3}E(\Gamma)$. Then $\Gamma$ is a minimal surface and $\partial\Gamma$ meets $\partial\Omega$ orthogonally (if at all). 
\end{lemma}
\begin{proof}
The criticality of $u_\e$ implies that $\delta E(u_\e,\eta)=0$ for all $C^3(\overline{\Omega})$ vector fields $\eta$. Now,
for any smooth vector field $\eta\in (C^3(\overline{\Omega}))^N$ such that $\eta\cdot \nu=0$ on $\partial\Omega$,
we have from Theorem \ref{thm-AC2} that
$$\int_{\Gamma} \div^{\Gamma}\eta\, d\mathcal{H}^{N-1}=\delta E(\Gamma,\eta)=\frac{3}{4}\lim_{\e\rightarrow 0} \delta E_{\e} (u_\e,\eta)=0.$$
We decompose $\eta=\eta^{\perp}+ \eta^{T}$ where $\eta^{\perp} =(\eta\cdot{\bf n}){\bf n}$. Then $\div^{\Gamma} \eta^{\perp}= (n-1)H(\eta\cdot {\bf n})$ where $H$ denotes the mean curvature of $\Gamma$. Now, we have
from the Divergence Theorem that
\begin{equation}
\label{Hzero}
0=\int_{\Gamma} \div^{\Gamma}\eta \,d\mathcal{H}^{N-1}=(n-1)\int_{\Gamma} H (\eta\cdot{\bf n})\,d\mathcal{H}^{N-1} + \int_{\p \Gamma \cap \p\Omega} \eta^{T}\cdot {\bf n}^\ast \,d\mathcal{H}^{N-2}
\end{equation}
where ${\bf n}^\ast$ is the outward unit co-normal of $\p\Gamma\cap\Omega$, that is, $n^\ast$ is normal to $\p\Gamma\cap \p\Omega$ and tangent to $\Gamma$.
First, we consider vector fields $\eta$ compactly supported in $\Omega$. From (\ref{Hzero}), we then obtain
$$\int_{\Gamma} H (\eta\cdot{\bf n})\,d\mathcal{H}^{N-1}=0 $$
for all $\eta \in (C^3_{0}(\Omega))^N$. This allows us to
conclude that
$H=0$ on $\Gamma$, that is, $\Gamma$ is a minimal surface. 

Now, using this new information and returning to (\ref{Hzero}), we find that
$$\int_{\p \Gamma \cap \p\Omega} \eta^{T}\cdot {\bf n}^\ast d\mathcal{H}^{N-2}$$
for all
smooth vector fields $\eta$ such that $\eta\cdot \nu=0$ on $\partial\Omega$. This implies that $\p\Gamma$ is orthogonal to $\p\Omega$ (see, for example, \cite[p. 70]{SZ2}).
\end{proof}
\subsection{Upper semicontinuity of the Neumann eigenvalues}
 Now we prove Theorem \ref{eigen_thm} concerning an asymptotic upper bound for the Neumann eigenvalues of the operators $-\e \Delta + 2\e^{-1}(3u_{\e}^2-1)$ in the limit $\e\to 0$ under appropriate conditions.

\begin{proof}[Proof of Theorem \ref{eigen_thm}]
The proof follows the argument of \cite[Corollary 1.1]{Le2}. We include its details for completeness.

Let denote by $Q_{\e}$ the quadratic function associated to the operator $-\e \Delta + 2\e^{-1}(3u_{\e}^2-1)$ with zero Neumann boundary conditions, that is, for $\varphi\in C^{1}(\overline{\Omega})$, we have
$$Q_{\e}(u_\e)(\varphi)=\int_{\Omega} \left(\e |\nabla \varphi|^2 + 2\e^{-1} (3u_\e^2-1)\varphi^2\right) dx\equiv d^2 E_{\e}(u_{\e}, -\varphi).$$
Similarly, for the Robin eigenvalue problem (\ref{Robinintro}), we can define a quadratic function $Q$ for the operator $-\Delta_{\Gamma} - |A_{\Gamma}|^2$ in $\Gamma$ with a Robin condition on $\partial\Gamma \cap\partial\Omega$ for the corresponding eigenfunctions for $-\Delta_{\Gamma} - |A_{\Gamma}|^2$. 
That is, for $\varphi\in C^{1}(\overline{\Gamma})$, we define
\[
Q(\varphi)= \int_{\Gamma} \left(\abs{\nabla^{\Gamma} \varphi}^2 - \abs{A}^2 \varphi^2\right) d\mathcal{H}^{N-1}-\int_{\partial\Gamma\cap\partial\Omega} A_{\partial\Omega}({\bf n}, {\bf n})|\varphi|^2 d\mathcal{H}^{N-2};
\]
see \cite[p. 398]{CH}.
We can naturally extend $Q$ to be defined for vector fields in $\overline{\Omega}$ that are generated by functions defined on $\overline{\Gamma}$ as follows. Given $f\in C^{1}(\overline{\Gamma})$, let $\eta = f {\bf n}$ be a normal vector field defined on $\Gamma$. 
Assuming the smoothness of $\overline{\Gamma}$, we deduce from Lemma \ref{ortho_lem} that  $\Gamma$ is a minimal surface and $\partial\Gamma$ meets $\partial\Omega$ orthogonally (if at all). Thus,
we can find an extension $\tilde{\eta}$ of $\eta$ to $\overline{\Omega}$ such that it is 
tangent to $\partial\Omega$, that is $\tilde \eta\cdot\nu=0$ on $\p\Omega$, $({\bf n}, 
{\bf n}\cdot\nabla\tilde{\eta}) =0$ on $\Gamma$. Then, define $Q(\tilde{\eta}):= Q(f).$

For any vector field $V$ defined on $\overline{\Gamma}$ and is normal to $\Gamma$, we also denote by $V$ its extension to $\overline{\Omega}$
in such a way that  it is 
tangent to $\partial\Omega$, $({\bf n}, {\bf n}\cdot \nabla V) =0$ on $\Gamma$. Let $\xi=\xi_V= V\cdot {\bf n}. $

Note that, using the stationarity of $u_\e$ and Corollary \ref{inner_rem}, we have for all vector field $\zeta$
$$Q_{\e}(\nabla u_{\e}\cdot V) = d^2 E_{\e}(u_{\e}, -\nabla u_{\e}\cdot V)= \delta^2 E_{\e}(u_{\e}, V, \zeta).$$
We choose
$$\zeta= (V\cdot \nabla) V.$$
Then, we have, by Theorems \ref{thm-AC2} and \ref{PI_ineq}
\begin{eqnarray}\lim_{\e\rightarrow 0} Q_{\e}(\nabla u_{\e}\cdot V) &=& \lim_{\e\rightarrow 0}  \delta^2 E_{\e}(u_{\e}, V, \zeta)\nonumber\\ &=& 
 \frac{4}{3}\int_{\Gamma} \left(|\nabla_{\Gamma}\xi|^2 -|A_{\Gamma}|^2|\xi|^2\right) d\mathcal{H}^{N-1} - \frac{4}{3}\int_{\partial\Gamma\cap\partial\Omega} A_{\partial\Omega}({\bf n}, {\bf n})|\xi|^2 d\mathcal{H}^{N-2}
 \nonumber\\ &=&   \frac{4}{3} Q(V).
\label{polarident}
\end{eqnarray}
By the definition of $\lambda_k$, we can find $k$ linearly independent, orthonormal vector fields $V^{1} = v^{1}{\bf n},\cdots, V^{k}= v^{k}{\bf n}$ which are defined on $\Gamma$ and normal to $\Gamma$ such that
\begin{equation}\label{V_ortho} 
\int_{\Gamma} v^i v^j d\mathcal{H}^{N-1}=\delta_{ij}~\text{and}~
\max_{\sum_{i=1}^{k} a^2_{i}=1} Q(\sum_{i=1}^{k} a_{i} V^{i})\leq \lambda_k. 
\end{equation}
Denote $$V^{i}_{\e}  = \left.\frac{d}{dt}\right\rvert_{t=0} u_{\e} \left(\left(x+ tV^{i}(x)\right)^{-1}\right)=-\nabla u_{\e}\cdot V^i.$$ 
As in \cite{Le}, we can use Lemma \ref{Res_Sp} to show that the map $V\longmapsto -\nabla u_{\e}\cdot V$ is linear and one-to-one for $\e$ small.  Thus, the linear independence of $V^{i}$ implies that of $V^{i}_{\e}$ for $\e$ small. Therefore, the $V^{i}_{\e}$ span a space of dimension $k$. It follows from the variational characterization of $\lambda_{\e, k}$ that
\begin{equation}\displaystyle
\sup_{\sum_{i=1}^{k} a^2_{i}=1} \frac{Q_{\e}(\sum_{i=1}^k a_i V^{i}_{\e})}{\e\int_{\Omega} |\sum_{i=1}^k a_i V^{i}_{\e}|^2}\geq \frac{\lambda_{\e, k}}{\e}.
\end{equation}
Take any sequence $\e\rightarrow 0$ such that
$$\frac{\lambda_{\e, k}}{\e}\rightarrow \limsup_{\e\rightarrow 0}\frac{\lambda_{\e, k}}{\e}:= \gamma_k.$$ Then, for any $\delta>0$, we can find $a_1, \cdots, a_k$ with $\sum_{i=1}^k a_i^2=1$ such that for $\e$ small enough
\begin{equation} \frac{Q_{\e}(\sum_{i=1}^k a_i V^{i}_{\e})}{\e\int_{\Omega} |\sum_{i=1}^k a_i V^{i}_{\e}|^2} \geq \gamma_k -\delta.
\label{Qeq1}
\end{equation}
By polarizing (\ref{polarident}) as in \cite{Le}, we have for all $a_{i}$
\begin{equation}
\lim_{\e\rightarrow 0}  Q_{\e} (\sum_{i=1}^{k} a_{i} V_{\e}^{i}) =  \frac{4}{3} Q (\sum_{i=1}^{k} a_{i} V^{i}).\label{Qeq2}
\end{equation}
and the convergence is uniform with respect to $\{a_{i}\}$ such that $\sum_{i=1}^{k} a^2_{i} =1$. 

Next, we study the convergence of the denominator of the left hand side of (\ref{Qeq1}) when $\e\rightarrow 0$.
Invoking Lemma \ref{Res_Sp}, we have
\begin{eqnarray}\lim_{\e\rightarrow 0} \e\int_{\Omega} |\sum_{i=1}^k a_i V^{i}_{\e}|^2dx &= &\lim_{\e\rightarrow 0} \e\int_{\Omega} \sum_{i, j=1}^k a_i a_j (\nabla u^{\e}\cdot V^i)(\nabla u^{\e}\cdot V^j)dx\nonumber \\
&=& \frac{4}{3}\sum_{i,j=1}^k a_i a_j \int_{\Gamma} v^i v^j d\mathcal{H}^{N-1}=\frac{4}{3},
\label{Qeq3}
\end{eqnarray}
where we used the first equation in (\ref{V_ortho}) in the last equation.
Combining (\ref{Qeq1})-(\ref{Qeq3}) together with (\ref{V_ortho}), we find that
$$ \gamma_k -\delta\leq Q(\sum_{i=1}^{k} a_{i} V^{i})\leq \lambda_k.$$
Therefore, by the arbitrariness of $\delta$, we have
$\gamma_k\leq \lambda_k,$ proving the theorem.
\end{proof}
\begin{rem}
If the hypotheses of
Theorem \ref{eigen_thm} are replaced by those of Theorem \ref{thm-ACc} then the upper semicontinuity of the Allen-Cahn Neumann eigenvalues still holds as stated in Theorem \ref{eigen_thm}. For this, we just replace the following in the above proof of Theorem \ref{eigen_thm}:
\begin{myindentpar}{1cm}
(i) the use of Theorem \ref{thm-AC2} by the use of Theorem \ref{thm-ACc};\\
(ii) the min-max characterization of eigenvalues by the {\bf weighted} min-max characterization of eigenvalues  as in \cite[Section 4]{Ga} and \cite[Section 3.2]{Hi};\\
(iii) Lemma \ref{Res_Sp} by (\ref{mKc}).
\end{myindentpar}
\end{rem}

\section{The inner variations of a nonlocal energy and their asymptotic limits }
\label{B_sec}
With the ultimate aim of studying the asymptotic limits of the Gateaux variations and  inner variations of the nonlocal Ohta-Kawasaki energy in the following section (see \eqref{OKa}), we turn now to the calculation and asymptotic behavior of these variations for the nonlocal part of this energy. To this end, for each $u\in L^1(\Omega)$, we denote its average on $\Omega$ by
\begin{equation}\bar{u}_{\Omega}:=\frac{1}{\abs{\Omega}}\int_{\Omega} u(x) dx.
\label{volc} 
\end{equation}
We associate to each $u\in L^2(\Omega)$ a function $v\in W^{2,2}(\Omega)$ as the solution to the following Poisson equation with Neumann boundary condition:
\begin{equation}
-\Delta v = u-\bar{u}_{\Omega}~\text{in}~\Omega, \frac{\partial v}{\partial \nu}=0~\text{on}~\partial\Omega, ~\int_{\Omega} v(x) dx=0.
\label{Poisson}
\end{equation}
Let $G(x, y)$ be the Green's function for $\Omega$ with the Neumann boundary condition:
\begin{equation}
\label{G_def}
-\Delta_y G(x, y)=\delta_x -\frac{1}{|\Omega|} ~\text{in }\Omega, ~\frac{\partial G(x,y)}{\partial \nu_y} =0~~\text{on }\partial\Omega,~ \int_{\Omega} G(x, y) dx =0 ~(\text{for all } y\in\Omega),
\end{equation} 
where $\delta_x$ is a delta-mass measure supported at $x\in\Omega$.

If ${\bf \Phi}(x)$ is the fundamental solution of Laplace's equation, that is,
\begin{equation*}
{\bf \Phi}(x):=\left\{ \begin{array}{ll}
         -\frac{1}{2\pi}\log |x|& \mbox{if $N =2$},\\
        \frac{1}{|B_1(0)|N (N-2) |x|^{N-2}} & \mbox{if $N >2$},\end{array} \right.
\end{equation*}
 then, for any fixed $x\in \Omega$,
 \begin{equation}
 \label{G_Phi}
 G(x, y)-{\bf \Phi} (x-y)~\text{is a } C^{\infty}~\text{function (of y) in a neigborhood of x}.
 \end{equation}
 Note that
$$v(x) = \int_{\Omega} G(x, y) u(y) dy.$$
Consider the following nonlocal functional on $L^{2}(\Omega)$
\begin{equation}
\label{B_def}
B(u):= \int_{\Omega}|\nabla v(x)|^2dx =\int_{\Omega}G(x, y) u(x) u(y) dxdy.
\end{equation}

The following lemma provides formulae for the Gateaux variations and inner variations of $B$ up to the second order.
\begin{lemma} [Gateaux variations and inner variations of $B$]
\label{B_lem}
Assume that $u\in C^3(\overline{\Omega})$, $\varphi\in C^3(\overline{\Omega})$ and $\eta,\zeta\in (C^3(\overline{\Omega}))^N$. Let $B(u)$ be defined as in (\ref{B_def}). Then, one has,
\begin{equation}
d B(u, \varphi) = 2\int_{\Omega}\int_{\Omega} G(x, y) u(y) \varphi(x) dxdy =2\int_{\Omega} v\varphi dx,
\end{equation}
\begin{equation}
d^2 B(u, \varphi) =  2\int_{\Omega}\int_{\Omega} G(x, y) \varphi(x) \varphi(y) dxdy,
\end{equation}
\begin{equation}
\label{oneB}
\delta B(u, \eta)=
-2\int_{\Omega}\int_{\Omega} G(x, y) u(y)\nabla u(x) \cdot \eta(x) dxdy,
\end{equation}
and
\begin{multline}
\delta^2 B(u, \eta, \zeta)= 2\int_{\Omega}\int_{\Omega} G(x, y) (\nabla u(y)\cdot \eta(y)) (\nabla u(x)\cdot \eta(x)) dxdy \\+ 2\int_{\Omega}\int_{\Omega} G(x, y) u(x)X_0(y) dxdy\label{twoB}
\end{multline}
where we recall from \eqref{Xzero} that
$$X_0=(D^2 u \cdot \eta, \eta) + (\nabla u, 2(\eta\cdot\nabla ) \eta -\zeta).$$
\end{lemma}
An immediate consequence of Lemma \ref{B_lem} is the following corollary which is a nonlocal counterpart of Corollary \ref{inner_rem}. It establishes the relationship between Gateaux variations and inner variations up to the second order.
\begin{cor}
\label{B_cor}
Assume that $u\in C^3(\overline{\Omega})$, and $\eta,\zeta\in (C^3(\overline{\Omega}))^N$. Let $B(u)$ be defined as in (\ref{B_def}). Then, one has,
\begin{equation}
\label{Bfirst}
\delta B(u, \eta) = d B(u, -\nabla u\cdot \eta),
\end{equation}
\begin{equation}\delta^2 B(u, \eta, \zeta) = d^2 B(u, -\nabla u\cdot \eta) + dB (u, X_0).  
\label{Bsecondgen}
\end{equation}
\end{cor}
\begin{proof}[Proof of Lemma \ref{B_lem}]
The formulae for $d B(u, \varphi)$ and $d^2 B(u, \varphi)$ can be obtained easily using their definitions 
\[
dB(u,\varphi)= \left.\frac{d}{dt}\right\rvert_{t=0} B(u + t\varphi),~d^{2} B(u,\varphi) = \left.\frac{d^2}{dt^2}\right\rvert_{t=0} B(u + t\varphi),\]
so we skip their derivations.

Now we establish the formulae for $\delta B(u,\eta,\zeta)$ and $\delta^2 B(u,\eta,\zeta)$.

Let  $\Phi_t(x)= x+ t\eta (x) + \frac{t^2}{2} \zeta(x)$ and $u_{ t}(y):= u(\Phi^{-1}_t(y))$.
Then, by (\ref{ut_expand}), we have
\begin{equation}
u_{ t}(y):= u(\Phi^{-1}_t(y)) = u (y) -t\nabla u(y)\cdot  \eta(y) + \frac{t^2}{2} X_0 (y)+ O(t^3).\label{utnew}
\end{equation}
It follows that 
\begin{multline*}
B(u_t)=\int_{\Omega}\int_{\Omega} G(x, y) u_{t}(y)u_{t}(x)dxdy= \int_{\Omega}G(x, y) u(x) u(y) dxdy\\ -2t\int_{\Omega}\int_{\Omega} G(x, y) u(y)\nabla u(x) \cdot \eta(x) dxdy\\ + t^2
\left(\int_{\Omega}\int_{\Omega} G(x, y) (\nabla u(y)\cdot \eta(y)) (\nabla u(x)\cdot \eta(x)) dxdy + \int_{\Omega}\int_{\Omega} G(x, y) u(x)X_0(y) dxdy\right) + O(t^3).
\end{multline*}
Recalling (see (\ref{inner1defn}) and (\ref{innerdefn2})) that
\begin{equation*}
\delta B(u,\eta) = \left.\frac{d}{dt}\right\rvert_{t=0} B(u_t),~
 \delta^{2} B(u,\eta,\zeta) = \left.\frac{d^2}{dt^2}\right\rvert_{t=0} B(u_t),
\end{equation*}
we obtain the first and second inner variations for $B$ as asserted.
\end{proof}

The next theorem  studies the asymptotic limits of the inner variations of the nonlocal functional $B$ under suitable assumptions. It can be viewed as a nonlocal analogue of Theorem \ref{thm-AC2}. As in this theorem, in order to pass to the limit the second inner variation 
$\delta^2 B(u_\e, \eta, \zeta)$, we can focus on a particular choice of the acceleration vector field $\zeta$. Instead of imposing $\zeta= Z:= (\eta\cdot\nabla)\eta$ as in Theorem \ref{thm-AC2}, we find that we can still pass to the limit
when the tangential parts of $\zeta$ and $Z$ coincide on the boundary $\partial\Omega$. 
\begin{thm}[Limits of inner variations of the nonlocal energy $B$]
\label{B_thm}
Let $\{u_{\e}\}\subset C^3 (\overline{\Omega})$ be a sequence of functions that converges in $L^2(\Omega)$ to a function $u_0\in BV (\Omega, \{1, -1\})$ with an interface $\Gamma=\partial\{u_0=1\}\cap\Omega$ having the property that $\overline{\Gamma}$ is $C^2$.  
Throughout, we will denote by ${\bf n}$ the unit normal to $\Gamma$ pointing out of the region $\{x:\,u_0(x)=1\}$. 
Let $G$ be defined as in (\ref{G_def}).
Let $B$ be defined as in (\ref{B_def}).
Then, for all smooth vector fields $\eta,\zeta\in (C^{3}(\overline{\Omega}))^{N}$ with $\eta\cdot \nu=0$ on $\partial\Omega$,  
and $\zeta\cdot \nu=Z\cdot\nu$ on $\partial\Omega$ where we recall $Z:= (\eta\cdot\nabla)\eta$
we 
have
\begin{equation}
\label{oneB_lim}
\lim_{\e\rightarrow 0}\delta B(u_\e, \eta)=  4 \int_{\Gamma} v_0(\eta\cdot {\bf n}) d \mathcal{H}^{N-1} (x).
\end{equation}
and 
\begin{multline}
\label{twoB_lim}
\lim_{\e\rightarrow 0}\delta^2 B(u_\e, \eta, \zeta)= 8\int_{\Gamma}\int_{\Gamma} G(x, y) (\eta(x)\cdot {\bf n}(x))(\eta(y)\cdot {\bf n}(y))  d \mathcal{H}^{N-1} (x) d \mathcal{H}^{N-1}(y) \\+ 4 \int_{\Gamma} (\nabla v_0\cdot\eta) (\eta\cdot {\bf n}) d \mathcal{H}^{N-1} (x) + 4 \int_{\Gamma} v_0 (\zeta-Z + (\div \eta)\eta)\cdot {\bf n} d \mathcal{H}^{N-1} (x).
\end{multline}
Here we use the following notations:
\[
v_\e(x) := \int_{\Omega} G(x, y) u_\e(y) dy\quad\mbox{and}\quad v_0(x) := \int_{\Omega} G(x, y) u_0(y) dy.
\]
\end{thm}

\begin{proof}[Proof of Theorem \ref{B_thm}]
We will apply Lemma \ref{B_lem} where $X_0$ is now replaced by
\begin{eqnarray}X_{\e}&:=& (D^2 u_\e \cdot \eta, \eta) + (\nabla u_\e, 2(\eta\cdot\nabla ) \eta -\zeta)\nonumber\\&=&(D^2 u_{\e} \cdot \eta, \eta) + (\nabla u_{\e}, (\eta\cdot\nabla ) \eta + \div (\eta) \eta)
+ (\nabla u_\e, (\eta\cdot\nabla)\eta- (\div\eta)\eta -\zeta)
\nonumber\\ &=& \div \left((\nabla u_\e\cdot \eta)\eta\right) +  (\nabla u_\e, Z-\zeta- (\div\eta)\eta)\equiv D_\e +  (\nabla u_\e, Z-\zeta- (\div\eta)\eta)
\label{Xep}
\end{eqnarray}
where
\begin{equation}
\label{Dep}
D_\e:=  \div \left((\nabla u_\e\cdot \eta)\eta\right).
\end{equation}
From (\ref{oneB}) in Lemma \ref{B_lem}, we have 
\begin{equation}
\delta B(u_\e, \eta)=
-2\int_{\Omega}\int_{\Omega} G(x, y) u_{\e}(y)(\nabla u_{\e}(x)\cdot \eta(x))\, dxdy\label{Bone}.
\end{equation}
From  (\ref{twoB}) in Lemma \ref{B_lem} together with (\ref{Xep}), we obtain
\begin{multline}
\delta^2 B(u_\e, \eta, \zeta)= 2\int_{\Omega}\int_{\Omega} G(x, y) (\nabla u_{\e}(y)\cdot \eta(y)) (\nabla u_{\e}(x)\cdot \eta(x)) dxdy \\+ 2\int_{\Omega}\int_{\Omega} G(x, y) u_{\e}(y)D_{\e}(x) dxdy+
2\int_{\Omega}\int_{\Omega} G(x, y) u_{\e}(y)(\nabla u_\e(x), Z(x)-\zeta(x)- (\div\eta(x))\eta(x)) dxdy.\label{Btwo}
\end{multline}
{\bf Claim 1:} We have
\begin{equation*}
\lim_{\e\rightarrow 0} \int_{\Omega}\int_{\Omega} G(x, y) u_{\e}(y)(\nabla u_{\e}(x)\cdot \eta(x)) dxdy= -2\int_{\Gamma} v_0(\eta\cdot {\bf n}) d \mathcal{H}^{N-1} (x)
\end{equation*}
and 
\begin{multline*}
\lim_{\e\rightarrow 0} \int_{\Omega}\int_{\Omega} G(x, y) u_{\e}(y)(\nabla u_\e(x), Z(x)-\zeta(x)- (\div\eta(x))\eta(x)) dxdy\\= -2 \int_{\Gamma} v_0 (Z-\zeta - (\div \eta)\eta)\cdot {\bf n} d \mathcal{H}^{N-1} (x).
\end{multline*}
{\bf Claim 2:} We have
\begin{align*}
\lim_{\e\rightarrow 0} \int_{\Omega}\int_{\Omega} &G(x, y) (\nabla u_{\e}(y)\cdot \eta(y)) (\nabla u_{\e}(x)\cdot \eta(x)) dxdy \\
&= 4\int_{\Gamma}\int_{\Gamma} G(x, y) (\eta(x)\cdot {\bf n}(x))(\eta(y)\cdot {\bf n}(y))   d \mathcal{H}^{N-1} (x) d \mathcal{H}^{N-1}(y).
\end{align*}
{\bf Claim 3:} For $D_\e$ as in (\ref{Dep}), we have
\[
\lim_{\e\rightarrow 0} \int_{\Omega} \int_{\Omega} G(x, y) u_{\e}(y)D_{\e}(x)\, dxdy=2\int_{\Gamma} (\nabla v_0\cdot\eta) (\eta\cdot {\bf n})\ d \mathcal{H}^{N-1}(x).
\]
Using the above claims in (\ref{Bone}) and (\ref{Btwo}), we obtain (\ref{oneB_lim}) and (\ref{twoB_lim}) as claimed in the theorem.

We now prove the above claims. 

Let us start with the proof of Claim 3. Using (\ref{Dep}) and $\eta\cdot\nu=0$ on $\partial\Omega$, we find after two integrations by parts that 
\begin{eqnarray}
\label{GDep1}
\int_{\Omega}\int_{\Omega}G(x, y) u_{\e}(y)D_{\e}(x) \,dxdy&=&
\int_{\Omega} v_{\e}(x) D_{\e}(x)\, dx =\int_{\Omega} v_{\e} \div ((\nabla u_{\e}\cdot \eta)\eta)\,dx\nonumber\\&=&-\int_{\Omega}(\nabla v_{\e}\cdot\eta)(\nabla u_{\e}\cdot\eta)\,dx= \int_{\Omega} \div \left((\nabla v_{\e}\cdot\eta)\eta\right) u_\e\, dx. \end{eqnarray}
From $u_\e\rightarrow u_0$ in $L^2(\Omega)$ and the global $W^{2,2}(\Omega)$ estimates for 
the Poisson equation (\ref{Poisson}),
we have
\begin{equation}
\label{vep_w22}
 v_\e\rightarrow v_0~\text{in } W^{2,2}(\Omega).
\end{equation}
In particular
$D^2 v_\e\rightarrow D^2 v_0~\text{in } L^{2}(\Omega).$
Thus, when $\e\rightarrow 0$, we have
\begin{equation}
\label{GDep2}
\int_{\Omega} \div \left((\nabla v_{\e}\cdot\eta)\eta\right) u_\e dx \rightarrow \int_{\Omega} \div \left((\nabla v_0\cdot\eta)\eta\right) u_0 dx=2\int_{\Gamma} (\nabla v_0\cdot\eta) (\eta\cdot {\bf n})\ d \mathcal{H}^{N-1}(x).
\end{equation}
Combining (\ref{GDep1}) and (\ref{GDep2}), we obtain Claim 3.

Let us now prove Claim 1. We start with the first limit. We have
\begin{equation*}
\int_{\Omega}\int_{\Omega} G(x, y) u_{\e}(y)(\nabla u_{\e}(x)\cdot \eta(x)) dxdy= \int_{\Omega} v_\e (x) (\nabla u_{\e}(x)\cdot \eta(x)) dx.
\end{equation*}
Integrating by parts and using the fact that $\eta\cdot\nu=0$ on $\partial\Omega$, we have
\begin{eqnarray*}
 \int_{\Omega} v_\e (x) (\nabla u_{\e}(x)\cdot \eta(x))dx=-\int_{\Omega} \div(v_\e\eta) u_\e dx &\rightarrow& -\int_{\Omega} \div(v_0\eta) u_0 dx\\&=& -2\int_{\Gamma} v_0(\eta\cdot {\bf n}) d \mathcal{H}^{N-1} (x).
\end{eqnarray*}
In the above convergence, we have used the facts that $u_\e\rightarrow u_0$ in $L^2(\Omega)$ and $\div(v_\e\eta)\rightarrow \div(v_0\eta)$ in $W^{1,2}(\Omega)$ which is a consequence of (\ref{vep_w22}). The first limit of Claim 1 is hence established. 
The proof of the second limit in Claim 1 is similar. Here we replace $\eta$ in the first limit by $Z-\zeta- (\div\eta)\eta$ in the second limit. For this, we note that from $\zeta\cdot\nu=Z\cdot\nu$ on $\partial\Omega$ and $\eta\cdot\nu=0$ on $\partial\Omega$, we also have
$(Z-\zeta- (\div\eta)\eta)\cdot\nu=0$ on $\partial\Omega$. 
The proof of Claim 1 is thus completed.

Finally, we prove Claim 2. 
To do this, we introduce some notations. Let
\begin{equation}
\label{aep}
a_\e(x)= \nabla u_\e(x)\cdot\eta (x) \in C^2 (\overline{\Omega}).
\end{equation}
Let $w_\e$ be the solution to the following Poisson equation with Neumann boundary condition:
$$-\Delta w_\e = a_\e-\bar{a_\e}_{\Omega}~\text{in}~\Omega, \frac{\partial w_\e}{\partial \nu}=0~\text{on}~\partial\Omega, ~\int_{\Omega} w_\e(x) dx=0.$$
Then $w_\e\in C^{3,\alpha}(\overline{\Omega})$ for all $\alpha\in (0, 1)$ and
\begin{equation}
\label{wep}
w_\e(x) = \int_{\Omega} G(x, y) a_\e(y) dy.
\end{equation}
Integrating by parts and using the fact that $\eta\cdot\nu=0$ on $\partial\Omega$, we have
\begin{eqnarray}
\label{GtwoU}
\int_{\Omega}\int_{\Omega} G(x, y) (\nabla u_{\e}(y)\cdot \eta(y)) (\nabla u_{\e}(x)\cdot \eta(x))\, dxdy&=& \int_{\Omega} w_\e(x) (\nabla u_{\e}(x)\cdot \eta(x)) \,dx
\nonumber \\&=&-  \int_{\Omega} \div(w_\e\eta) u_{\e} dx.
\end{eqnarray}
To prove Claim 2, we study the convergence property  in $L^p(\Omega)$ of  $w_\e$ and $\nabla w_\e$. 

Integrating by parts and using the fact that $\eta\cdot\nu=0$ on $\partial\Omega$, we have from (\ref{aep}) and (\ref{wep})
\begin{equation}
\label{wep2}
w_\e(x) =  \int_{\Omega} G(x, y) \nabla u_\e(y)\cdot\eta (y) dy =-\int_{\Omega} \div_y (G(x, y)\eta(y)) u_\e(y) dy.
\end{equation}
Using (\ref{G_Phi}), we find that the most singular term in $\div_y (G(x, y)\eta(y))$ is of the form $\frac{C}{|x-y|^{N-1}}$ which, for a fixed $x$,  belongs to $L^p(\Omega)$ for all $p<\frac{N}{N-1}$. Thus, when $u_\e\in L^2(\Omega)$, we have by Young's convolution inequality that
$w_\e\in L^ q$ for all $q<q_\ast=\frac{2N}{N-2}$ which comes from the relation
$$\frac{1}{q_\ast} + 1= \frac{N-1}{N} +\frac{1}{2}.$$
In particular, if $u_\e\rightarrow u_0$ in $L^2(\Omega)$ then from (\ref{wep2}), we have the following convergence in $L^q(\Omega)$ for all $q<\frac{2N}{N-2}$:
\begin{equation}w_\e\rightarrow w_0.
\label{wlq}
\end{equation}
where
\begin{equation}
\label{w0for}
w_0(x):= -\int_{\Omega} \div_y (G(x, y)\eta(y)) u_0(y) dy =-2\int_{\Gamma} G(x, y)\eta (y) \cdot {\bf n}(y) d \mathcal{H}^{N-1}(y).
\end{equation}

For the convergence of $\nabla w_\e$, we observe from (\ref{wep2}) that
\begin{equation}
\label{dwep}
\nabla w_\e(x) = -\int_{\Omega} \div_y (\nabla_x G(x, y)\eta(y)) u_\e(y) dy.
\end{equation}
Expanding $\div_y (\nabla_x G(x, y)\eta(y))$ and using (\ref{G_Phi}), we find that the most singular term on the right hand side of (\ref{dwep}) is of the form $$R_{ij} (\eta u_\e)(x):=\int_{\Omega} \frac{(x_i-y_i)(x_j-y_j)}{|x-y|^{N+2}}\eta(y) u_\e(y) dy.$$ 
Applying the $L^2-L^2$ estimates in Calderon-Zygmund theory of singular integral operators, we find that
$$\|R_{ij}(\eta u_\e)\|_{L^2(\Omega)} \leq C(N,\Omega) \|\eta u_\e\|_{L^{2}(\Omega)}.$$
It follows that, if $u_\e\rightarrow u_0$ in $L^2(\Omega)$ then we have the following convergence in $L^2(\Omega)$:
\begin{equation}\nabla w_\e(x)\rightarrow \nabla w_0(x) = -\int_{\Omega} \div_y (\nabla_xG(x, y)\eta(y)) u_0(y) dx.
\label{dwl2}
\end{equation}
From (\ref{wlq}) and (\ref{dwl2}), we have
\begin{equation}
\label{wu_lim}
-  \int_{\Omega} \div(w_\e\eta) u_{\e} dx\rightarrow -\int_{\Omega} \div(w_0\eta) u_0 dx.
\end{equation}
Using (\ref{w0for}), we find that
\begin{eqnarray}
\label{wu0}
-\int_{\Omega} \div(w_0\eta) u_0 dx&=& -2\int_{\Gamma} w_0(x) \eta (x)\cdot {\bf n}(x)  d \mathcal{H}^{N-1} (x)\nonumber \\
&=&
4\int_{\Gamma}\int_{\Gamma} G(x, y) (\eta(x)\cdot {\bf n}(x))(\eta(y)\cdot {\bf n}(y))   d \mathcal{H}^{N-1} (x) d \mathcal{H}^{N-1}(y).
\end{eqnarray}
Combining (\ref{GtwoU}), (\ref{wu_lim}) and (\ref{wu0}), we obtain the limit as asserted in Claim 2. This completes the proof of Claim 2 and also the proof of our theorem.
\end{proof}

\section{Applications of Second Variation Convergence for Ohta-Kawasaki}
\label{OK_sec}
We now wish to analyze the asymptotic behavior of the inner first and second variations of the Ohta-Kawasaki functional
\begin{equation} \mathcal{E}_{\e,\gamma}(u)= E_{\varepsilon}(u) + \frac{4}{3}\gamma B(u)=\int_{\Omega}\left(\frac{\varepsilon \abs{\nabla u}^2}{2} +\frac{(1-u^2)^2}{2\varepsilon}\right) dx + \frac{4}{3}\gamma\int_{\Omega}|\nabla v|^2 dx,\label{OKa}\end{equation}
a model for microphase separation in diblock copolymers; see \cite{OK}. 
Here $\e>0$ and $\gamma\geq 0$, $u: \Omega\rightarrow \RR$ and we are using the same notation for $B$ as in (\ref{B_def}) so that $v$ is required to satisfy \eqref{Poisson}.  The factor of $\frac{4}{3}$ is simply put in for convenience in stating the $\Gamma$-convergence result. These functionals are known to $\Gamma$-converge in $L^1(\Omega)$ to $\frac{4}{3}\mathcal{E}_{\gamma}$
 where
\begin{equation}
\mathcal{E}_{\gamma}(u_0):= E(u_0) +\gamma\, B(u_0)\label{Egamma},
\end{equation}
(see \cite{RW}) where we recall that  
\begin{equation*}
E(u_0)=\left\{
 \begin{alignedat}{1}
   \frac{1}{2}\int_{\Omega}|\nabla u_0| ~&~ \text{if} ~u_0\in BV (\Omega, \{1, -1\}), \\\
\infty~&~ \text{otherwise}.
 \end{alignedat} 
  \right.
  \end{equation*}
  As in Section \ref{AC_sec}, 
  if the interface $\Gamma:= \partial\{x\in \Omega: u_0(x)=1\}\cap \Omega$ separating 
the $\pm1$ phases of $u_0\in BV (\Omega, \{1, -1\})$ is sufficiently regular, say $C^1$, then  we also identify $$E(u_0)\equiv E(\Gamma)=\mathcal{H}^{N-1}(\Gamma)$$
and
\begin{equation}
\label{Egamma2}
\mathcal{E}_{\gamma}(u_0) \equiv \mathcal{E}_{\gamma}(\Gamma)=  E(\Gamma) + \gamma B(u_0) = \mathcal{H}^{N-1}(\Gamma) + \gamma B(u_0).
\end{equation}
  
  Competitors $u:\Omega\to\R$ in the Ohta-Kawasaki functional are generally required to satisfy a mass constraint 
  \begin{equation}
  \label{mass}
  \frac{1}{|\Omega|}\int_{\Omega}u\,dx=m\quad\mbox{for some constant}\; m\in (-1,1).
  \end{equation}
  We should mention that all of the analysis of this section applies, in particular, to the special case where $\gamma=0$, that is to the case of the mass-constrained Allen-Cahn or Modica-Mortola functionals. Under such a constraint this context is perhaps better known as the equilibrium setting for the Cahn-Hilliard problem.
  
We first establish the following theorem which is the nonlocal Ohta-Kawasaki analogue of Theorem \ref{thm-AC2}. It allows us to pass the the limit the first and second inner variations of the Ohta-Kawasaki functionals, without imposing any criticality conditions.
\begin{thm}[Limits of inner variations of the Ohta-Kawasaki functional]
\label{SIV_OK}
Let $ \mathcal{E}_{\e,\gamma}$ and $ \mathcal{E}_{\gamma}$ be as in (\ref{OKa}) and (\ref{Egamma2}) respectively. Let $G$ be defined as in (\ref{G_def}).
Let $\{u_{\e}\}\subset C^3 (\overline{\Omega})$ be a sequence of functions that converges in $L^2(\Omega)$ to a function $u_0\in BV (\Omega, \{1, -1\})$ with an interface $\Gamma=\partial\{u_0=1\}\cap\Omega$ having the property that $\overline{\Gamma}$ is $C^2$.  Assume that
 $\lim_{\varepsilon \rightarrow 0} \mathcal{E}_{\e,\gamma}(u_{\varepsilon}) = 
\frac{4}{3}\mathcal{E}_\gamma(\Gamma).$ Let $v_0(x) := \int_{\Omega} G(x, y) u_0(y) dy.$
Then, for all smooth vector fields $\eta\in (C^{3}(\overline{\Omega}))^{N}$ with $\eta\cdot \nu=0$ on $\partial\Omega$, we 
have
\begin{equation}
\label{1vm_lim}
\lim_{\varepsilon\rightarrow 0}\delta\mathcal{E}_{\e,\gamma}(u_{\varepsilon}, \eta) = \frac{4}{3}\left(\delta E(\Gamma,\eta)+ 4 \gamma\int_{\Gamma}  v_0 (\eta\cdot {\bf n}) d \mathcal{H}^{N-1} \right)
\end{equation} 
and for such $\eta$ and for $\zeta\in (C^{3}(\overline{\Omega}))^{N}$ with $\zeta\cdot \nu=Z\cdot\nu$ on $\partial\Omega$ where  $Z=(\eta\cdot\nabla)\eta$,
we 
have
\begin{multline}
\label{2vm_lim}
\lim_{\varepsilon\rightarrow 0}\frac{3}{4}\delta^{2} \mathcal{E}_{\e,\gamma}(u_{\varepsilon}, \eta, \zeta) = \delta^2 E(\Gamma,\eta, Z) + \int_{\Gamma} ({\bf n},{\bf n}\cdot\nabla\eta)^2 d\mathcal{H}^{N-1} +  \int_{\Gamma} \div^{\Gamma} (\zeta-Z) d\mathcal{H}^{N-1}\\ +  8\gamma\int_{\Gamma}\int_{\Gamma} G(x, y) (\eta(x)\cdot {\bf n}(x))(\eta(y)\cdot {\bf n}(y))  d \mathcal{H}^{N-1} (x) d \mathcal{H}^{N-1}(y) \\+ 4 \gamma\int_{\Gamma} (\nabla v_0\cdot\eta) (\eta\cdot {\bf n}) d \mathcal{H}^{N-1} 
 + 4 \gamma\int_{\Gamma} v_0 (\zeta-Z + (\div \eta)\eta)\cdot {\bf n} d \mathcal{H}^{N-1}.
\end{multline} 
\end{thm}
\begin{proof}
Let $B(u)$ be defined as in (\ref{B_def}).
First, note that from (\ref{OKa}), (\ref{Egamma}) and $\lim_{\varepsilon \rightarrow 0} \mathcal{E}_{\e,\gamma}(u_{\varepsilon}) = 
\frac{4}{3}\mathcal{E}_\gamma(\Gamma),$ we also have
$$\lim_{\varepsilon \rightarrow 0} E_{\e}(u_{\varepsilon}) = 
\frac{4}{3}E(\Gamma),$$
since  the $L^2(\Omega)$-convergence of $\{u_\e\}$ to $u_0$ implies that $B(u_\e)\rightarrow B(u_0).$ This means that all conditions of Theorems \ref{thm-AC2} and \ref{B_thm}
are satisfied and we can apply their results to the proof of our theorem.
 
Next, observe that
\begin{equation*}
\delta \mathcal{E}_{\e,\gamma}(u_{\varepsilon}, \eta) = \delta E_{\e}(u_{\varepsilon}, \eta) + \frac{4}{3}\gamma\,\delta  B(u_{\varepsilon}, \eta).
\end{equation*}
Therefore, (\ref{1vm_lim}) follows from Theorems \ref{thm-AC2} and \ref{B_thm}.

Turning to the proof of (\ref{2vm_lim}), we have from the definition of $\mathcal{E}_{\e,\gamma}$ in (\ref{OKa}) that
\begin{equation*}
\frac{3}{4}\delta^{2} \mathcal{E}_{\e,\gamma}(u_{\varepsilon}, \eta, \zeta) = \frac{3}{4}\delta^{2} E_{\e}(u_{\varepsilon}, \eta, \zeta) + \gamma\,\delta^{2} B(u_{\varepsilon}, \eta, \zeta).
\end{equation*}
We now apply \eqref{secondgen} to $E_{\e}$ at $u_\e$, first with $X_0$ given by (\ref{Xzero}) with $\zeta$ itself and then with $\zeta=Z$ and subtract to find that
\begin{eqnarray*}
\delta^{2} E_{\e}(u_{\varepsilon}, \eta, \zeta) &=& \delta^{2} E_{\e}(u_{\varepsilon}, \eta, Z) + dE_{\e}(u_\e, \nabla u_\e \cdot (Z-\zeta))\\ &=&  \delta^{2} E_{\e}(u_{\varepsilon}, \eta, Z) + \delta E_{\e}(u_\e, \zeta-Z).
\end{eqnarray*}
In the last equation, we have used (\ref{first}) relating the first Gateaux variation and the first inner variation.
It follows that
\begin{equation}
\label{E_split}
\frac{3}{4}\delta^{2} \mathcal{E}_{\e,\gamma}(u_{\varepsilon}, \eta, \zeta) = \frac{3}{4}\left( \delta^{2} E_{\e}(u_{\varepsilon}, \eta, Z) + \delta E_{\e}(u_\e, \zeta-Z)\right) +\gamma\delta^{2} B(u_{\varepsilon}, \eta, \zeta)
\end{equation}
Letting $\e\rightarrow 0$ in $\delta E_{\e}(u_\e, \zeta-Z)$, we find from Theorem \ref{thm-AC2} and \eqref{FVE} that
\begin{equation*}
\lim_{\e\rightarrow 0}  \frac{3}{4}\delta E_{\e}(u_\e, \zeta-Z) = \delta E(\Gamma,\zeta-Z)= \int_{\Gamma} \div^{\Gamma} (\zeta-Z) d\mathcal{H}^{N-1}.
\end{equation*}
Letting $\e\rightarrow 0$ in (\ref{E_split}), using the above limit together with Theorems \ref{thm-AC2} and \ref{B_thm}, we obtain (\ref{2vm_lim}).
\end{proof}

Next we wish to apply Theorem \ref{SIV_OK} to the case of stable critical points of the Ohta-Kawasaki functional  $\mathcal{E}_{\e,\gamma}$ subject to a mass constraint
which is the context of Theorem \ref{OK2}.
 To be clear, we refer to a function $u:\Omega\to\R$ as a critical point of $\mathcal{E}_{\e,\gamma}$ subject to a mass constraint if
 $d\mathcal{E}_{\e,\gamma}(u,\phi)=0$ whenever $\int_\Omega\phi(y)\,dy=0$, and we say $u$ is a stable critical point of the Ohta-Kawasaki functional $ \mathcal{E}_{\e,\gamma}$ if additionally 
$d^2\mathcal{E}_{\e,\gamma}(u,\phi)\geq 0$ for such functions $\phi.$

Before proving Theorem \ref{OK2}, we would like to explain the peculiar choices of the velocity and acceleration vector fields $\eta$ and $\zeta$ stated in the theorem. Their choices were explained in \cite[Theorem 1.4]{Le2}. For reader's convenience, we repeat the argument here in the following remark.
\begin{rem}
\label{W_rem}
The choice of the velocity and acceleration vector fields $\eta$ and $\zeta$ in
$$\Phi_t(x)= x+ t\eta(x) +\frac{t^2}{2}\zeta(x)$$
in applications to the inner variations of the mass-constrained Ohta-Kawasaki functional is motivated by the fact that we wish the family 
$\Phi_t(E_0)$ of deformations of $E_0:=\{x\in\Omega: u_0(x)=1\}$ 
to preserve the volume of $E_0$ up to the second order in $t$, that is, 
\begin{equation}
\label{vol_Et}
|\Phi_t(E_0)|= |E_0| + o(t^2).
\end{equation}
 For $t$ sufficiently small, we have as in (2.16),
\begin{multline*}
\abs{\det\nabla\Phi_{t}(x)}=\det\nabla\Phi_{t}(x) =\det (I + t\nabla \eta (x) +\frac{t^2}{2}\nabla \zeta)\\= 1 +  t\div \eta + \frac{t^2}{2}[ \div\zeta + (\div\eta)^2 - \text{trace}((\nabla\eta)^2)] + O(t^3).
\end{multline*}
It follows that, for small $t$, we have
$$|\Phi_t(E_0)| = \int_{E_0} \abs{\det\nabla\Phi_{t}(x)}dx= \int_{E_0} \{1 +  t\div \eta + \frac{t^2}{2}[ \div\zeta + (\div\eta)^2 - \text{trace}\,((\nabla\eta)^2)] + O(t^3)\} \,dx .$$
The requirement (\ref{vol_Et}) is reduced to a set of two equations:
\begin{equation}
\label{vol_Et2}
\int_{E_0} \text{div}\, \eta~ dx =0,~\text{and}~\int_{E_0} [ \text{div}\,\zeta + (\text{div}\,\eta)^2 - \text{trace}\,((\nabla\eta)^2)]\,dx=0.
\end{equation}
Note that 
$$ (\div\,\eta)^2 - \text{trace}\,((\nabla\eta)^2) = \div \left((\div\,\eta)\eta - (\eta\cdot\nabla)\eta\right).$$
Thus, for any $\eta$, we can choose
$\zeta = W:= - (div\eta)\eta + (\eta\cdot\nabla)\eta$
so that the second equation in (\ref{vol_Et2}) holds. The issue is now reduced to the first equation in (\ref{vol_Et2}).
 However, when $\int_{\Gamma}\eta\cdot{\bf n}d\mathcal{H}^{n-1}=0$ and $\eta\cdot\nu=0$ on $\p\Omega$, an application of the divergence theorem shows that the first equation is also satisfied. 
\end{rem}
We can now present the proof of Theorem \ref{OK2} from the introduction.
\begin{proof}[Proof of Theorem \ref{OK2}]

Consider smooth vector fields $\eta\in (C^{3}(\overline{\Omega}))^{N}$ 
 satisfying 
 \begin{equation}
 \int_\Gamma \eta \cdot {\bf n}d\mathcal{H}^{N-1}(x)=0\quad\mbox{and}\quad \eta\cdot \nu=0\;\mbox{on}\; \partial\Omega.\label{gdeta}
 \end{equation}
 As explained in Remark \ref{W_rem}, (\ref{gdeta}) guarantees the preservation of mass up to the first order for the limit problem if we deform the set
 $E_0:=\{x\in\Omega: u_0(x)=1\}$ using $\Phi_t(E_0)$ where $\Phi_t(x) = x + t\eta (x) + O(t^2).$ 
 Furthermore, with (\ref{gdeta}) in hand, we can choose the acceleration vector field $\zeta:=W= (\eta\cdot\nabla )\eta -(\div\eta)\eta$ so that 
 if we deform the set
 $E_0$ using $\Phi_t(E_0)$ where $\Phi_t(x) = x + t\eta (x) + \frac{t^2}{2}\zeta (x) + O(t^3),$ the mass is preserved up to second order.

 Now, we ``lift'' all these to the $\e$-level.
 
Our first task will be to create a perturbation of $u_\e$  in the form of 
\begin{equation}
\label{ue_mass}
u_{\e, t}(y)=u_\e(\Phi_{\e, t}^{-1}(y))
\end{equation}
that preserves the mass constraint \eqref{mass} to second order
for a suitable deformation map 
$$\Phi_{\e, t}(y)= y + t\eta^{\e}(y) + \frac{t^2}{2} \zeta^{\e}(y) + O(t^3).$$
To this end,
first we construct $C^3(\overline{\Omega})$ perturbations $\eta^{\e}$ of $\eta$ such that $\eta^\e\cdot\nu=0$ on $\p\Omega$, and
\begin{equation}
\label{etae}
\lim_{\e\rightarrow 0}\|\eta^{\e}-\eta\|_{C^{3}(\overline{\Omega})}=0,\quad\int_{\Omega} u_{\e}\div \eta^{\e}dx=0.
\end{equation}
 In light of \eqref{ue_mass} and \eqref{ut_expand} with $\eta$ replaced by $\eta^\e$, the integral condition in \eqref{etae} will guarantee that to first order, mass is conserved since
\begin{equation}
\left.\frac{d}{dt}\right\rvert_{t=0} \int_{\Omega}u_{\e, t}(y)\,dy=-\int_{\Omega}\nabla u_\e\cdot \eta^\e\,dy=0.\label{zeromean}
\end{equation}

Here is a simple way to construct $\eta^{\e}$. Choose any smooth vector field $\beta\in (C^{3}(\overline{\Omega}))^{N}$ satisfying $\beta\cdot\nu=0$ on $\p\Omega$ and
$\int_{\Gamma}\beta\cdot {\bf n} d\mathcal{H}^{N-1}(x)\neq 0.$
Let
$$h(\e):=\frac{-\int_{\Omega} u_{\e}\div \eta \,dx}{\int_{\Omega} u_{\e}\div \beta \,dx}~\text{and}~\eta^{\e}= \eta (x) + h(\e) \beta(x).$$
Then, the second equation in (\ref{etae})
 is satisfied.  Let
$E_0=\{x\in \Omega: u_0(x)=1\}.$ Then, as $\e\rightarrow 0,$ we have
$$h(\e)\rightarrow \frac{-2\int_{E_0}\div\,\eta\, dx}{2\int_{E_0} \div\,\beta \,dx}= \frac{-2\int_{\Gamma}\eta\cdot {\bf n}\, d\mathcal{H}^{N-1}(x)}{2\int_{\Gamma}\varphi\cdot {\bf n}\,d\mathcal{H}^{N-1}(x)}=0.$$
Therefore, the first equation in (\ref{etae}) is also satisfied. 

With \eqref{zeromean} in hand,  the function  $-\nabla u_\e \cdot\eta^\e$ is admissible in computing the first and second Gateaux variations of $\mathcal{E}_{\e,\gamma}$ with respect to the mass constraint \eqref{mass}. We will first investigate the $\e\rightarrow 0$ limit of the criticality condition
$d \mathcal{E}_{\e,\gamma} (u_{\e}, -\nabla u^{\e}\cdot\eta^{\e})=0$ 

(i) Using the convergence of $\eta^\e$ to $\eta$ given in \eqref{etae}, along with the uniform boundedness of 
$\mathcal{E}_{\e,\gamma}(u_{\e})$, a glance at the explicit formulae for 
$\delta \mathcal{E}_{\e,\gamma} (u_{\e}, \eta^{\e})= \delta E_{\e}(u_\e,\eta^\e) + \frac{4}{3}\gamma \delta B(u_\e,\eta^\e)$ given in \eqref{star} and \eqref{Bone} easily leads to the conclusion that
\begin{equation*}
\lim_{\e\rightarrow 0}\delta \mathcal{E}_{\e,\gamma} (u_{\e}, \eta^{\e})= \lim_{\e\rightarrow 0}\delta \mathcal{E}_{\e,\gamma} (u_{\e}, \eta)
=\frac{4}{3}\left(\delta E(\Gamma,\eta )+ 4 \gamma\int_{\Gamma}  v_0 (\eta\cdot {\bf n}) d \mathcal{H}^{N-1} (x)\right),
\end{equation*}
where the last equality comes from \eqref{1vm_lim} of Theorem \ref{SIV_OK}.
Using (\ref{first}) and  (\ref{Bfirst}), we have
\begin{equation*}
d \mathcal{E}_{\e,\gamma} (u_{\e}, -\nabla u^{\e}\cdot\eta^{\e})= \delta \mathcal{E}_{\e,\gamma} (u_{\e}, \eta^{\e}).
\end{equation*}
Combining the above equations with $d \mathcal{E}_{\e,\gamma} (u_{\e}, -\nabla u^{\e}\cdot\eta^{\e})=0$, we get
\begin{equation*}
\delta E(\Gamma,\eta )+ 4 \gamma\int_{\Gamma}  v_0 (\eta\cdot {\bf n})\, d \mathcal{H}^{N-1} (x)=0.
\end{equation*}
Invoking \eqref{FVE} we find that 
\begin{equation}
\int_{\Gamma} (\div^{\Gamma}\eta + 4 \gamma  v_0 (\eta\cdot {\bf n})) \,d \mathcal{H}^{N-1} (x)=0.
\end{equation}
By decomposing $\eta= \eta^{\perp} +\eta^{T}$ where $\eta^{\perp}= (\eta\cdot{\bf n}){\bf n}$, we have
\[
0=\int_{\Gamma}((n-1) H + 4 \gamma v_0)(\eta\cdot{\bf n})\,d \mathcal{H}^{N-1} + \int_{\p\Gamma\cap\p\Omega} \eta^{T}\cdot {\bf n}^{\ast} d \mathcal{H}^{N-2}.
\]
Here we have used the Divergence Theorem to evaluate $\int_{\Gamma} \div^{\Gamma}\eta $ as in (\ref{Hzero}), and ${\bf n}^{\ast}$ denotes the co-normal vector orthogonal to $\partial\Omega\cap\partial\Gamma.$
Since this relation holds for all $\eta$ satsifying \eqref{gdeta}, it follows that  there is a constant $\lambda$ such that $(n-1) H + 4 \gamma v_0 =\lambda$ on $\Gamma$ and $\p\Gamma$ must meet $\p\Omega$ orthogonally, if at all. (See \cite[p. 70]{SZ2} for more details.) Thus, (i) is established.

(ii) Turning to the proof of (ii) we introduce
\[
W:= (\eta \cdot\nabla) \eta-(\div \eta)\eta\quad\mbox{and}\quad W^\e:=(\eta^\e \cdot\nabla) \eta^\e-(\div \eta^\e)\eta^\e.
\]
In light of the $C^3$ convergence of $\eta^\e$ to $\eta$ we note that 
\[
\lim_{\e\rightarrow 0}\|W^{\e}-W\|_{C^{2}(\overline{\Omega})}=0.
\]
Consequently, the uniform energy bound on $\mathcal{E}_{\e,\gamma}(u_{\e})$ and the explicit formulae for
$\delta^2 \mathcal{E}_{\e,\gamma} (u_{\e}, \eta^{\e}, W^{\e})=  \delta^2 E_{\e}(u_\e,\eta^\e) + \frac{4}{3}\gamma \delta^2 B(u_\e,\eta^\e)$ given in \eqref{svep-p} and \eqref{Btwo} imply that
\begin{equation}\lim_{\e\rightarrow 0}\delta^2 \mathcal{E}_{\e,\gamma} (u_{\e}, \eta^{\e}, W^{\e})= \lim_{\e\rightarrow 0}\delta^2 \mathcal{E}_{\e,\gamma} (u_{\e}, \eta, W).
\label{2svme}
\end{equation}
Now using the relation between the Gateaux and inner second variation of $E_\e$ and $B$ provided by Corollaries \ref{inner_rem} and \ref{B_cor}, we obtain
\begin{equation}d^2 \mathcal{E}_{\e,\gamma} (u_{\e}, -\nabla u_{\e}\cdot\eta^{\e})=\delta^2 \mathcal{E}_{\e,\gamma}(u_{\e}, \eta^{\e}, W^{\e})-d\mathcal{E}_{\e,\gamma}(u_{\e}, X_{\e})
\label{2svm}
\end{equation}
where 
\[
X_{\e}= (D^2 u_{\e}(y) \cdot \eta^{\e}(y), \eta^{\e} (y)) + (\nabla u_{\e}(y), (\eta^{\e}\cdot\nabla ) \eta^{\e} (y) + \div (\eta^{\e}) \eta^{\e})= \div ((\nabla u_{\e}\cdot \eta^{\e})\eta^{\e}).
\]
But since $\eta^\e\cdot\nu=0$ on $\p\Omega$, the divergence theorem implies that $\int_\Omega X_\e\,dx=0$ and so by the criticality of 
$u_\e$ we have $d\mathcal{E}_{\e,\gamma}(u_{\e}, X_{\e})=0$. The fact that the integral of $X_\e$ vanishes is no coincidence. It is precisely related to the fact that our choice of $W$ and of $W^\e$ preserve mass to second order.
The first order preservation was already guaranteed by  \eqref{zeromean}. For the second order preservation, we note that with $u_{\e, t}$ defined by (\ref{ue_mass}), we can use (\ref{ut_expand}) with 
 with $X_0$ replaced by $X_\e$ to get
$$
\left.\frac{d^2}{dt^2}\right\rvert_{t=0} \int_{\Omega}u_{\e, t}(y)\,dy=\int_{\Omega}X_\e(y)\,dy =0.$$

At this point we further restrict $\eta$ to additionally satisfy
\[
\eta=\xi {\bf n} \quad\mbox{and}\quad ({\bf n},{\bf n}\cdot\nabla\eta) =0\quad \mbox{on} \;\Gamma
\]
for any smooth function $\xi:\overline{\Omega}\rightarrow\RR$ satisfying
$$
\int_{\Gamma} \xi(x) d\mathcal{H}^{N-1}(x)=0.$$
From (\ref{2svme}) and (\ref{2svm})  together with Theorems  \ref{PI_ineq} and \ref{SIV_OK}, noting that $W-Z=-(\div \eta)\eta$, we obtain 
\begin{multline}\frac{3}{4}\lim_{\e\rightarrow 0}d^2 \mathcal{E}_{\e,\gamma} (u_{\e},  -\nabla u_{\e}\cdot\eta^{\e}) = \frac{3}{4}\lim_{\e\rightarrow 0}\delta^2 \mathcal{E}_{\e,\gamma} (u_{\e}, \eta, W)=
\delta^2 E(\Gamma,\eta, Z)\\- \int_{\Gamma} \div^{\Gamma} ((\div\eta)\eta) d\mathcal{H}^{N-1}+  8\gamma\int_{\Gamma}\int_{\Gamma} G(x, y) (\eta(x)\cdot {\bf n}(x))(\eta(y)\cdot {\bf n}(y))  d \mathcal{H}^{N-1} (x) d \mathcal{H}^{N-1}(y) \\+ 4 \gamma\int_{\Gamma} (\nabla v_0\cdot\eta) (\eta\cdot {\bf n}) d \mathcal{H}^{N-1} 
\\= \int_{\Gamma} \left(|\nabla_{\Gamma}\xi|^2 + (N-1)^2H^2\xi^2 -|A_{\Gamma}|^2|\xi|^2\right) d\mathcal{H}^{N-1} - \int_{\partial\Gamma\cap\partial\Omega} A_{\partial\Omega}({\bf n}, {\bf n})|\xi|^2 d\mathcal{H}^{N-2}
\\- \int_{\Gamma} \div^{\Gamma} ((\div\eta)\eta) d\mathcal{H}^{N-1}+  8\gamma\int_{\Gamma}\int_{\Gamma} G(x, y) \xi(x) \xi(y) d \mathcal{H}^{N-1} (x) d \mathcal{H}^{N-1}(y) + 4 \gamma\int_{\Gamma} (\nabla v_0\cdot {\bf n}) \xi^2d \mathcal{H}^{N-1}\\= \delta^2 \mathcal{E}_{\gamma} (\Gamma,\xi) +\int_{\Gamma} \left[ (N-1)^2H^2\xi^2-\div^{\Gamma} ((\div\eta)\eta)\right] d\mathcal{H}^{N-1}(x) .
\label{2svmlim}
\end{multline}
Using  $({\bf n},{\bf n}\cdot\nabla\eta) =0$  on $\Gamma$, we find that
$\div \eta = \div^\Gamma \eta =(N-1)H \xi$ on $\Gamma$.
Thus, on $\Gamma$ we have
$$\div^{\Gamma}  ((\div\eta)\eta) = \div^{\Gamma} ((N-1)H \xi^2 {\bf n})= (N-1)^2 H^2\xi^2.$$
Therefore, we get from (\ref{2svmlim}) the following limit
\begin{equation}
\label{2svmlim2}
\frac{3}{4}\lim_{\e\rightarrow 0}d^2 \mathcal{E} (u_{\e},  -\nabla u_{\e}\cdot\eta^{\e}) = \frac{3}{4}\lim_{\e\rightarrow 0}\delta^2 \mathcal{E}_{\e,\gamma} (u_{\e}, \eta, W)
=\delta^2 \mathcal{E}_{\gamma} (\Gamma,\xi).
\end{equation}
The proof of (\ref{secst}) is complete.

(iii) From (i), we know that $\p\Gamma$ must meet $\p\Omega$ orthogonally, if at all.  Thus, for any smooth function $\xi:\overline{\Omega}\to\R$ satisfying $\int_{\Gamma} \xi(x) d\mathcal{H}^{N-1}(x)=0$, we can choose 
a smooth vector field $\eta$ on $\overline{\Omega}$ such that $\eta=\xi{\bf n}$ on $\Gamma$, $\eta\cdot \nu=0$ on $\partial\Omega$ and such that $({\bf n}, {\bf n}\cdot\nabla \eta)=0$ on $\Gamma$. Let $\eta^\e$ be as in the proofs of (i) and (ii). Then, 
the stability inequality $\delta^2 \mathcal{E}_{\gamma} (\Gamma,\xi)\geq 0$ follows immediately from the limit (\ref{2svmlim2}) above, since $d^2 \mathcal{E}_{\e,\gamma} (u_{\e},  -\nabla u_{\e}\cdot\eta^{\e})\geq 0$ by the stability of $u_\e$. 

(iv) The proof is similar to that of
 Theorem \ref{eigen_thm}.  The most crucial point in the proof of Theorem \ref{eigen_thm} is the identity (\ref{polarident}) between two quadratic forms $Q_\e$ and $Q$ associated with the two eigenvalue problems. Now, in our
 nonlocal context, we will also obtain a similar identity (\ref{polarident2}). 
 
 To do so, we first set up the corresponing quadratic forms for our two eigenvalue problems. 
 Let denote by $\mathcal{Q}_{\e,\gamma}(u_\e)$ the quadratic function associated to the operator $$-\e \Delta + 2\e^{-1}(3u_{\e}^2-1) +\frac{8}{3}\gamma (-\Delta)^{-1}$$ with zero Neumann boundary conditions, that is, for $\varphi\in C^{1}(\overline{\Omega})$, we have
\begin{eqnarray*}\mathcal{Q}_{\e,\gamma}(u_\e)(\varphi)&=&\int_{\Omega} \left(\e |\nabla \varphi|^2 +  \frac{1}{\e}(6u_\e^2-2)\varphi^2\right) dx +\frac{8}{3}\gamma\int_{\Omega}\int_{\Omega} G(x, y)\varphi(x)\varphi(y) dxdy\\ &\equiv& d^2 \mathcal{E}_{\e,\gamma}(u_{\e}, -\varphi).\end{eqnarray*}
Similarly, we can define a quadratic function $\mathcal{Q}_{\gamma}$ for the operator 
\[
-\Delta_{\Gamma} - |A_{\Gamma}|^2+ 8\gamma (-\Delta)^{-1}(\chi_{\Gamma}) + 4\gamma (\nabla v_0\cdot{\bf n})
\] on $\Gamma$ with a Robin condition on $\partial\Gamma \cap\partial\Omega$ for the corresponding eigenfunctions. 
That is, for $\varphi\in C^{1}(\overline{\Gamma})$, we define
\begin{multline*}\mathcal{Q}_{\gamma}(\varphi)= \int_{\Gamma} \left(\abs{\nabla^{\Gamma} \varphi}^2 - \abs{A}^2 \varphi^2\right) d\mathcal{H}^{N-1}-\int_{\partial\Gamma\cap\partial\Omega} A_{\partial\Omega}({\bf n}, {\bf n})|\varphi|^2 d\mathcal{H}^{N-2}\\+  8\gamma\int_{\Gamma}\int_{\Gamma} G(x, y) \varphi(x) \varphi(y) d \mathcal{H}^{N-1} (x) d \mathcal{H}^{N-1}(y) + 4 \gamma\int_{\Gamma} (\nabla v_0\cdot {\bf n}) \varphi^2d \mathcal{H}^{N-1}.
\end{multline*}
We can naturally extend $\mathcal{Q}_{\gamma}$ to be defined for vector fields in $\overline{\Omega}$ that are generated by functions defined on $\overline{\Gamma}$ as follows. Given $f\in C^{1}(\overline{\Gamma})$, let $\eta = f {\bf n}$ be a normal vector field defined on $\Gamma$. 
Assuming the smoothness of $\Gamma$, we know from (i) that $\partial\Gamma$ must meet $\partial\Omega$ orthogonally (if at all). Thus, we can find an extension $\tilde{\eta}$ of $\eta$ to $\overline{\Omega}$ such that it is 
tangent to $\partial\Omega$, that is $\tilde\eta\cdot\nu=0$ on $\p\Omega$, $({\bf n}, 
{\bf n}\cdot\nabla\tilde{\eta}) =0$ on $\Gamma$. Then, define $\mathcal{Q}_{\gamma}(\tilde{\eta}):= \mathcal{Q}_{\gamma}(f).$

For any vector field $V$ defined on $\overline{\Gamma}$ that is normal to $\Gamma$, we also denote by $V$ its extension to $\overline{\Omega}$
in such a way that  it is 
tangent to $\partial\Omega$, $({\bf n}, {\bf n}\cdot \nabla V) =0$ on $\Gamma$. Let $\xi=\xi_V= V\cdot {\bf n}. $\\
Note that, using the stationarity of $u_\e$ with respect to a mass constraint, and (\ref{2svm}), we have for  $$\zeta:=(V\cdot \nabla) V-(\div V) V,$$
the identity
$$\mathcal{Q}_{\e,\gamma}(\nabla u_{\e}\cdot V) = d^2 \mathcal{E}_{\e,\gamma}(u_{\e},- \nabla u_{\e}\cdot V)= \delta^2 \mathcal{E}_{\e,\gamma}(u_{\e}, V, \zeta).$$
Then, we have, by (ii)
\begin{eqnarray}\lim_{\e\rightarrow 0} \mathcal{Q}_{\e,\gamma}(\nabla u_{\e}\cdot V) = \lim_{\e\rightarrow 0}  \delta^2 \mathcal{E}_{\e,\gamma}(u_{\e}, V, \zeta)&=&\frac{4}{3}\delta^2\mathcal{E}_{\gamma}(\Gamma, \xi)\nonumber\\ 
 &=&   \frac{4}{3} \mathcal{Q}_{\gamma}(V).
\label{polarident2}
\end{eqnarray}
Now, arguing similarly as in the proof of Theorem \ref{eigen_thm} starting right after (\ref{polarident}), we obtain the desired result.

(v) The proof of this part is similar to that of (iv). In fact, it is simpler. We use the argument in (iv) for functions $f$ and vector fields $V$ compactly supported in $\Gamma$.
\end{proof}

{\bf Acknowledgements.} The authors would like to thank the anonymous referee for the careful reading of the paper together with his/her constructive comments.

{} 

\end{document}